\documentclass[12pt]{article}
\usepackage[amsmath]{e-jc}
\usepackage{graphicx}
\usepackage[enableskew]{youngtab}
\usepackage{tikz}
\usetikzlibrary{positioning}
\usepackage[all,cmtip]{xy}
\usepackage{color}
\usepackage{indentfirst}

\newcommand\encircle[1]{%
  \tikz[baseline=(X.base)] 
    \node (X) [draw, shape=circle, inner sep=0] {\strut #1};}

\newcommand\enred[1]{%
  \tikz[baseline=(X.base)] 
    \node (X) [draw, shape=circle, inner sep=0, color=red] {\strut #1};}

\newcommand\enblue[1]{%
  \tikz[baseline=(X.base)] 
    \node (X) [draw, shape=circle, inner sep=0, color=blue] {\strut #1};}

\makeatother

\newcommand{\la}{\lambda}
\newcommand{\tb}{\textbf}

\newcommand{\cd}{\cdots}
\newcommand{\Z}{\mathbb{Z}}

\makeatletter

\newsavebox\myboxA
\newsavebox\myboxB
\newlength\mylenA

\newcommand*\xoverline[2][0.75]{%
    \sbox{\myboxA}{$\m@th#2$}%
    \setbox\myboxB\null
    \ht\myboxB=\ht\myboxA%
    \dp\myboxB=\dp\myboxA%
    \wd\myboxB=#1\wd\myboxA
    \sbox\myboxB{$\m@th\overline{\copy\myboxB}$}
    \setlength\mylenA{\the\wd\myboxA}
    \addtolength\mylenA{-\the\wd\myboxB}%
    \ifdim\wd\myboxB<\wd\myboxA%
       \rlap{\hskip 0.5\mylenA\usebox\myboxB}{\usebox\myboxA}%
    \else
        \hskip -0.5\mylenA\rlap{\usebox\myboxA}{\hskip 0.5\mylenA\usebox\myboxB}%
    \fi}
\makeatother

\dateline{Sep 5, 2023}{Feb 29, 2024}{TBD}

\MSC{05A15,05A17}

\Copyright{The authors. Released under the CC BY-ND license (International 4.0).}

\title{Proof of a Conjecture of Nath and Sellers on Simultaneous Core Partitions}

\author{Yetong Sha 
\and 
Huan Xiong
}

\authortext{}{Institute for Advanced Study in Mathematics, Harbin Institute of Technology, Heilongjiang 150001, China
(\email{dellenudaubg@gmail.com},\email{huan.xiong.math@gmail.com}).
}

\begin{document}




\maketitle

\begin{abstract}
In 2016, Nath and Sellers proposed a conjecture regarding the precise largest size of ${(s,ms-1,ms+1)}$-core partitions. In this paper, we prove their conjecture. One of the key techniques in our proof is to introduce and study the properties of generalized-$\beta$-sets, which extend the concept of $\beta$-sets for core partitions. Our results can be interpreted as a generalization of the well-known result of Yang, Zhong, and Zhou concerning the largest size of $(s,s+1,s+2)$-core partitions.
\end{abstract}

\section{Introduction} Recall that an \emph{integer partition}, or simply a \emph{partition}, is a finite non-increasing sequence of positive integers $\lambda = (\lambda_1, \lambda_2, \ldots, \lambda_\ell)$ with $\lambda_1 \geq \lambda_2\geq  \ldots\geq  \lambda_\ell>0$ (see \cite{Macdonald,ec2}). Here $\ell$ is called the \emph{length}, $\la_i\ (1\leq i\leq \ell)$ are the \emph{parts} and $|\la|:=\sum_{1\leq i\leq \ell}\lambda_i$ is the \emph{size} of $\lambda$. Each partition $\lambda$ can be visualized by its \emph{Young diagram}, which is an array of boxes arranged in left-justified rows with $\lambda_i$ boxes in the $i$-th row. By flipping the Young diagram of the partition along its main diagonal, we obtain another partition corresponding to the new Young diagram. Such partitions are said to be \emph{conjugate} to each other. A partition is called \emph{self-conjugate} if it is equal to its conjugate partition.
For each box $\square=(i,j)$ in the $i$-th row and the $j$-th column of the Young diagram of $\la$, its \emph{hook length} $h_\square=h_{ij}$ is defined to be the number of boxes exactly below, exactly to the right, or the box itself. Let~$s>0$ be a positive integer.  A partition $\lambda$ is called an \emph{$s$-core partition} if its hook length set doesn't contain any multiple of~$s$. Furthermore,  $\lambda$ is called an \emph{$(s_1,s_2,\ldots, s_m)$}-core partition if it is simultaneously an $s_1$-core, an $s_2$-core, $\ldots$, and an $s_m$-core partition (see \cite{tamd, KN}). For instance, Figure \ref{fig:1} gives the Young diagram and hook lengths of the partition $(6,3,2,1)$ and Figure \ref{fig:2} gives the Young diagram and hook lengths of its conjugation. The partition $(6,3,2,1)$ is a $(4,6,11)$-core partition since its hook length set doesn't contain multiples of $4$, $6$ or $11$. One can note the first column hook lengths are $9,5,3,1$ and they
uniquely determine the partition.

\begin{figure}[htbp] \label{6321}
\begin{center}
\Yvcentermath1

$\young(975321,531,31,1)$


\end{center}
\caption{The Young diagram and hook lengths of the partition
$(6,3,2,1)$.} \label{fig:1}
\end{figure}

\begin{figure}[htbp] \label{432111}
\begin{center}
    $\young(9531,731,51,3,2,1)$
\end{center}
\caption{The Young diagram and hook lengths of the conjugation $(4,3,2,1,1,1)$ of the partition $(6,3,2,1)$.}
\label{fig:2}
\end{figure}

Core partitions arise naturally in the study of modular representation theory and combinatorics. 
For example, core partitions label the blocks of irreducible characters of symmetric groups (see \cite{ols}). Therefore, simultaneous core partitions play important roles in the study of relations between different blocks in the modular group representation theory.
Simultaneous core partitions are connected with rational combinatorics (see \cite{HN}).
Also, simultaneous core partitions are connected with Motzkin paths and Dyck paths (see \cite{HJJ1,HJ,SYH,SDH}).
Some statistics of simultaneous core partitions, such as numbers of partitions, numbers of corners, largest sizes and average sizes, have attracted much attention in the past twenty years (see  \cite{tamd1,  and, AHJ, Straub1, CHW,HJJ, CEZ, ford, HHL, PJ,NK,WX,NY,N3,NS,   SNKC,  SZ,Wang,Xiong2, Xiong4,   Za,Za2,ZZ}). For example, Anderson \cite{and} showed that the number of $(s_1,s_2)$-core partitions is equal to $\frac{1}{s_1+s_2}  \binom{s_1+s_2}{s_1}$ when $s_1$ and $s_2$ are coprime to each other. Armstrong \cite{AHJ} conjectured that the average size of $(s_1,s_2)$-core partitions equals~${(s_1-1)(s_2-1)(s_1+s_2+1)}/{24}$ when $s_1$ and $s_2$ are coprime to each other, which was first proved by Johnson \cite{PJ} and later by Wang \cite{Wang}. However, there are still a lot of unsolved problems in this field.

In this paper, we focus on the largest size of simultaneous core partitions. For $(s_1,s_2)$-core partitions, Olsson and Stanton \cite{ols} showed that the largest size of such partitions is~$\frac{({s_1}^2-1)({s_2}^2-1)}{24}$ when $s_1$ and $s_2$ are coprime to each other.  Straub \cite{Straub} studied the largest size of $(s,s+2)$-core partitions with distinct parts, and conjectured that the largest size should be equal to $\frac{1}{384}(s^2-1)(s+3)(5s+17)$ in 2016. Straub's conjecture was first proved by Yan, Qin, Jin and Zhou \cite{YQJZ}. Later, the second author \cite{Xiong3} obtained the largest sizes of~$(t, mt+1)$ and~$(t, mt-1)$-core partitions with distinct parts. Recently, Nam and Yu \cite{NY} derived formulas for the largest sizes of $(s,s+1)$-core partitions whose all parts are odd or all parts are even. 

For $m\geq 3$, \emph{$(s_1,s_2,\ldots, s_m)$}-core partitions have also been widely studied. The largest size of $(s,s+1,s+2)$-core partitions was conjectured to be 
$$
\begin{cases}
t\binom{t+1}{3} & \text{ if } s=2t-1;\\
t\binom{t}{3} +\binom{t+1}{3} & \text{ if } s=2t-2
\end{cases}
$$
by Amdeberhan \cite[Conjecture 11.2]{tamd}, and later proved by Yang, Zhong and Zhou \cite{YZZ}. Furthermore, the second author \cite{Xiong1} derived the formula for the largest size of general $(s,s+1,s+2,\ldots,s+p)$-core partitions. In 2019, Baek, Nam and Yu \cite{BNY} studied self-conjugate $(s,s+1,s+2)$-core partitions and obtained their largest size. 


In this paper, we prove the following result, which verifies Nath and Sellers' conjecture~\cite{NS2} on the largest size of  $(s,ms-1,ms+1)$-core partitions. 

\begin{theorem}[see Conjecture $57$ of \cite{NS2}]\label{conj} The largest size of an $(s,ms-1,ms+1)$-core partition is
$$
\begin{cases}
\frac{m^2t(t-1)(t^2-t+1)}{6} & \text{ if } s=2t-1;\vspace{1ex}\\
\frac{m^2(t-1)^2(t^2-2t+3)}{6} - \frac{m(t-1)^2}{2}  & \text{ if } s=2t-2.
\end{cases}
$$
There are two such maximal partitions; one's $\beta$-set is ${\mathcal L}_m(s)$ (please see Section 2 for the definition of $\beta$-sets and Section 3 for the definition of ${\mathcal L}_m(s)$), and the other one is the conjugate of the first one.
\end{theorem}

\begin{remark} Theorem \ref{conj} can be seen as a generalization of the well-known result of Yang, Zhong and Zhou \cite{YZZ} on the largest size of $(s, s + 1, s + 2)$-core partitions. Thus in the following discussion, we could assume that $s, m \ge 2$.
\end{remark}

\begin{remark}
    Nath and Sellers \cite{NS2} computed the size of the longest $(s,ms-1,ms+1)$-core partition, i.e. the partition with the largest $\beta$-set. Our result shows that the longest of $(s,ms-1,ms+1)$-core partitions is also the largest one.
\end{remark}



Next, we provide two examples for Theorem \ref{conj}.

\begin{example}\label{ex:1.2}
(1) When $s=5$ and $m=3$, Figures \ref{L_3(5)-2} and \ref{P_2} show the two $\beta$-sets of maximal $(s,ms-1,ms+1)$-core partitions (here the $\beta$-set is just the collection of first column hook lengths, visually represented as beads on the $s$-abacus, see Definition \ref{beta} for the formal definition of the $\beta$-set). Figure \ref{L_3(5)-2} corresponds to the partition $(12,9,9,6,6,3,3,3,2,2,2,2,1,1,1,1)$. Figure \ref{P_2} corresponds to the partition $(16,12,8,5,5,\\ 5,3,3,3,1,1,1)$. Both of them have the size $63$ and are conjugate to each other. In each figure, the $\beta$-set is shown by the circled positions.

(2) When $s=6$ and $m=3$, Figures \ref{L_3(6)-2} and \ref{P_3'} show the two $\beta$-sets of maximal $(s,ms-1,ms+1)$-core partitions. Figure \ref{L_3(6)-2} corresponds to the partition~$(22,17,12,12,9,9,9,6,6,\\6,3,3,3,3,2,2,2,2,2,1,1,1,1,1)$. Figure \ref{P_3'} corresponds to the partition $(24,19,14,10,10,\\ 10,7,7,7,4,4,4,2,2,2,2,2,1,1,1,1,1)$. Both of them have the size~$135$ and are conjugate to each other. 
\end{example}

\begin{figure}[h!]
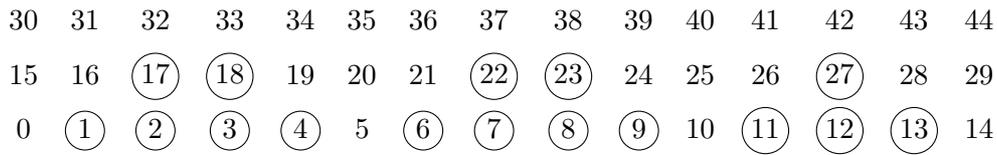

    \centering
    \[ \small
\begin{array}{ccccccccccccccc}
30 & 31 & 32 & 33 & 34 & 35 & 36 & 37 & 38 & 39 & 40 & 41 & 42 & 43 & 44\\
15 & 16 & \encircle{17} & \encircle{18} & 19 & 20 & 21 & \encircle{22} & \encircle{23} & 24 & 25 & 26 & \encircle{27} & 28 & 29\\
0 & \encircle{1} & \encircle{2} & \encircle{3} & \encircle{4} & 5 & \encircle{6} & \encircle{7} & \encircle{8} & \encircle{9} & 10 & \encircle{11} & \encircle{12} & \encircle{13} & 14\\
\end{array} 
\]
    \caption{$\mathcal {L}_3(5)$: The $\beta$-set of a maximal $(5,14,16)$-core partition.}
    \label{L_3(5)-2}
\end{figure}

\begin{figure}[h!]
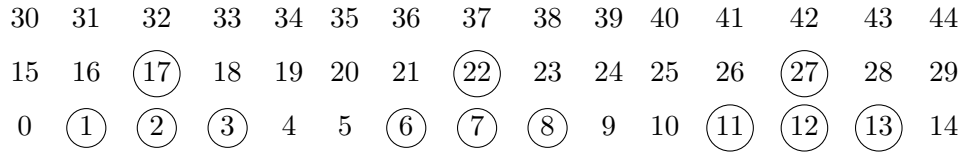

    \centering
    \[ \small
\begin{array}{ccccccccccccccc}
30 & 31 & 32 & 33 & 34 & 35 & 36 & 37 & 38 & 39 & 40 & 41 & 42 & 43 & 44\\
15 & 16 & \encircle{17} & 18 & 19 & 20 & 21 & \encircle{22} & 23 & 24 & 25 & 26 & \encircle{27} & 28 & 29\\
0 & \encircle{1} & \encircle{2} & \encircle{3} & 4 & 5 & \encircle{6} & \encircle{7} & \encircle{8} & 9 & 10 & \encircle{11} & \encircle{12} & \encircle{13} & 14\\
\end{array} 
\]
    \caption{The $\beta$-set of the other maximal $(5,14,16)$-core partition.}
    \label{P_2}
\end{figure}

\begin{figure}[h!]
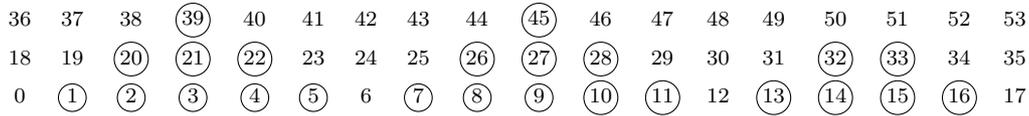

    \centering
    $$ \scriptsize 
    \begin{array}{cccccccccccccccccc}
    36 & 37 & 38 & \encircle{39} & 40 & 41 & 42 & 43 & 44 & \encircle{45} & 46 & 47 & 48 & 49 & 50 & 51 & 52 & 53 \\
    18 & 19 & \encircle{20} & \encircle{21} & \encircle{22} & 23 & 24 & 25 & \encircle{26} & \encircle{27} & \encircle{28} & 29 & 30 & 31 & \encircle{32} & \encircle{33} & 34 & 35\\
    0 & \encircle{1} & \encircle{2} & \encircle{3} & \encircle{4} & \encircle{5} & 6 & \encircle{7} & \encircle{8} & \encircle{9} & \encircle{10} & \encircle{11} & 12 & \encircle{13} & \encircle{14} & \encircle{15} & \encircle{16} & 17\\
    \end{array}
   $$
    \caption{$\mathcal {L}_3(6)$: The $\beta$-set of a maximal $(6,17,19)$-core partition.}
    \label{L_3(6)-2}
\end{figure}

\begin{figure}[h!]
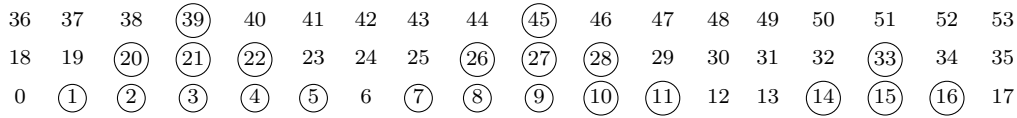

    \centering
   \[ \scriptsize 
    \begin{array}{cccccccccccccccccc}
    36 & 37 & 38 & \encircle{39} & 40 & 41 & 42 & 43 & 44 & \encircle{45} & 46 & 47 & 48 & 49 & 50 & 51 & 52 & 53 \\
    18 & 19 & \encircle{20} & \encircle{21} & \encircle{22} & 23 & 24 & 25 & \encircle{26} & \encircle{27} & \encircle{28} & 29 & 30 & 31 & 32 & \encircle{33} & 34 & 35\\
    0 & \encircle{1} & \encircle{2} & \encircle{3} & \encircle{4} & \encircle{5} & 6 & \encircle{7} & \encircle{8} & \encircle{9} & \encircle{10} & \encircle{11} & 12 & 13 & \encircle{14} & \encircle{15} & \encircle{16} & 17\\ 
    \end{array}
    \]
    \caption{The $\beta$-set of the other maximal $(6,17,19)$-core partition.}
    \label{P_3'}
\end{figure}

\par The rest of the paper is structured as follows. In Section $2$, we recall the definitions of $\beta$-sets and $s$-abacus diagrams, which serve as basic tools to study simultaneous core partitions. 
Following \cite{NS2}, we present the structure of $ms$-abacus diagrams of  $\beta$-sets of $(s,ms-1,ms+1)$-core partitions in Section~$3$. Next, we define and study the properties of generalized-$\beta$-sets in Section $4$. This will allow us to loosen the restriction of $\beta$-sets to make room for the adjustments we will need.
In Section $5$, we prove several lemmas that will be used repeatedly in the later sections. In the following sections, we will use adjustments step by step to find the necessary conditions of $S$ for $S$ being the generalized-$\beta$-set that maximizes $f(S)$ (the definition of $f(S)$ is given in Definition \ref{count}). 
In Section $6$, we determine the possible shape of each row of $S$. In Section $7$, we determine the possible shape of the entire $S$. Finally, we prove the main theorem in Section~$8$.

\section{$\beta$-sets and $s$-abacus diagrams}
In this section, we recall the definitions of $\beta$-sets and $s$-abacus diagrams for partitions.
\begin{definition}[\cite{berge,ols,Olsson}] \label{beta}
The $\beta$-set $\beta(\lambda)$ of a partition $\lambda=(\lambda_1,\cdots,\lambda_\ell)$ is denoted by
\[
\beta(\lambda):=\{h_{i1}:1\leq i\leq \ell\},
\]which contains hook lengths of boxes in the first column of the corresponding Young diagram of $\la$.
\end{definition}

\begin{remark}
It is easy to see that the $\beta$-set completely determines the partition.
\end{remark}

\begin{example}
The $\beta$-set of the partition $(6,3,2,1)$ is $\{9,5,3,1\}$ (See Figure \ref{6321}).
\end{example}



We introduce the notion of the size-counting function. It can be used to compute the size of a partition with the knowledge of its $\beta$-set.
\begin{definition}[size-counting function, see \cite{James}] \label{count}
Let $S$ be a finite set of integers. We define
\[
f(S):=\sum_{x\in S}~ x -\frac{|S|(| S|-1)}{2},
\]
where $|S|$ is the cardinality of the set $S$.
\end{definition}

\begin{lemma}[see \cite{Olsson, Xiong1}]
Let $S$ be the $\beta$-set of a partition $\lambda$. Then the size of~$\la$ is equal to $f(S)$, i.e.,
$$
|\la| = f(S).
$$
\end{lemma}

\par From now on we will focus on core partitions. The $\beta$-set of an $s$-core partition has the following property.

\begin{lemma}[see \cite{Olsson, Xiong1}]
Let $s$ be a positive integer and $S$ be the $\beta$-set of a partition $\lambda$. Then $\lambda$ is an~$s$-core partition, if and only if for any $x\in S$ and $x\ge s$, we have $x-s\in S$.
\end{lemma}

Next, we recall the definition of $s$-abacus. 

\begin{definition}[s-abacus  (see \cite{James,NS2})]
Let $X\subseteq   \mathbb{N}$ be a set of positive integers and $s$ be a positive integer. Then the $s$-abacus diagram 
$S$
of $X$ is defined to be the set 
$$
S:=\{(i,j): ~~ i\ge 0, ~~ 0\leq j\leq s-1, ~~  si+j\in X\}.
$$
\end{definition}

\begin{remark}
Usually, we use a diagram like Figure~\ref{2,4,9} to represent an $s$-abacus. When $(i,j)\in S$, the number in the $i$-th row and $j$-th column is circled. We begin with Row $0$ at the bottom and Column $0$ at the far left.

\end{remark}

\begin{example}
Figure \ref{2,4,9} shows the 5-abacus diagram of $\{2,4,9\}$, which is equal to 
$$
S=\{(0,2),~(0,4),~(1,4)\}.
$$
\begin{figure}[h!]
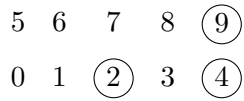

    \centering
    \[ \small
    \begin{array}{ccccc}
    5 & 6 & 7 & 8 & \encircle{9}  \\
    0 & 1 & \encircle{2} & 3 & \encircle{4} 
    \end{array}
    \]
    \caption{The 5-abacus diagram of \{2,4,9\}.}
    \label{2,4,9}
\end{figure}
\end{example}

\begin{remark}
The $s$-abacus is just a graphic representation of the $\beta$-set, 
thus we will use the terms $\beta$-set and $s$-abacus interchangeably. When we deal with the $ms$-abacus, we will write $is+j$ instead of $(i,j)$ for $i\ge 0$ and $0\leq j\leq s-1$. When we mention a certain row of a given set, we mean the corresponding row in the $ms$-abacus of the set.
\end{remark}

\section{The $\boldsymbol{\beta}$-sets of $\boldsymbol{(s,ms-1,ms+1)}$-core partitions}

In this section, we study the $\beta$-sets of $(s,ms-1,ms+1)$-core partitions. First, we recall the definition of s-pyramids.

\begin{definition}[s-pyramid (see \cite{NS2})]
Let $S$ be an $s$-abacus. Assume that $a$ and $b$ are two integers satisfying $0\le a<b\le s-1$. We call $S$ an s-pyramid with base $[a,b]$, if 
$$
S=\{      (i,j): ~~i\geq 0,~~~ a+i\leq j\leq b-i      \}.
$$
\end{definition}

\par For positive integers $n$, $a$ and $b$, let $[n]:=\{1,2,\ldots,n\}$ and $[a,b]:=\{a,a+1,a+2,\ldots,b-2,b-1,b\}$. The following theorem characterizes the shape of the $\beta$-set of an $(s,ms-1,ms+1)$-core partition.

\begin{theorem}[see Lemmas $1.6$ and $5.13$ of \cite{NS2}] \label{L_m(s)}
For given positive integers $s$ and $m$, let $P_k$ be an ms-pyramid with base $[(k-1)s+1, ks-1]$ for $1\leq k\leq m-1$ and $P_m$ be an ms-pyramid with base $[(m-1)s+1,ms-2]$. Furthermore, let $\mathcal {L}_m(s)=\cup_{k=1}^{m} P_k$ be the union of~$P_1, P_2, \ldots, P_m$. Then for any set $S$, $S$ is the $ms$-abacus of the $\beta$-set of an~$(s,ms-1,ms+1)$-core partition, if and only if it satisfies: 
\begin{enumerate}
    \item[(1)] $S\subseteq  \mathcal {L}_m(s)$;
    \item[(2)] If $(i,j)\in S$ and $i\ge 1$, then $(i-1,j-1)\in S$ and $(i-1,j+1)\in S$;
    \item[(3)] If $(i,j)\in S$ and $j\ge s$, then $(i,j-s)\in S$;
    \item[(4)] If $(i,j)\in S$, $i\ge 1$ and $j<s$, then $(i-1,j+(m-1)s)\in S$.
\end{enumerate}
\end{theorem}

\begin{example}
(1) When $m=3$ and $s=5$, we have $\mathcal {L}_3(5)=P_1\cup P_2 \cup P_3$ where $P_1=\{1,2,3,4,17,18\}$, $P_2=\{6,7,8,9,22,23\}$ and $P_3=\{11,12,13,27\}$. Figure \ref{L_3(5)-2} shows $\mathcal {L}_3(5)$. 

(2) When $m=3$ and $s=6$, we have $\mathcal {L}_3(6)=P_1\cup P_2 \cup P_3$ where the three $18$-pyramids are $P_1=\{1,2,3,4,5,20,21,22,29\}$, $P_2=\{7,8,9,10,11,26,27,28,45\}$ and $P_3=\{13,14,15,16,32,33\}$. Figure \ref{L_3(6)-2} shows $\mathcal {L}_3(6)$. 

 (3) Consider the $15$-abacus of the $\beta$-set $S$ of a $(5,14,16)$-core partition. Assume that $23,27\in S$. By Theorem \ref{L_m(s)}, we have $11,13,22\in S$ since $27\in S$; and $7,9,18\in S$ since $23\in S$. Repeating this process, we can deduce that $\mathcal{L}_3(5)\subseteq S$. Therefore, $S=\mathcal{L}_3(5)$.


\end{example}

\begin{lemma}\label{conjugate}
    For any positive integer $s$ and $m$, the partition with the beta set $\mathcal{L}_m(s)$ is not self-conjugate.
\end{lemma}

\begin{proof}
    We prove it by contradiction. Otherwise, the partition $\la=(\la_1,\la_2,\cdots,\la_u)$ that $\mathcal{L}_m(s)$ corresponds to is self-conjugate. Then $\la_1=u$, here $u=|\mathcal{L}_m(s)|$. Thus we have $\max\mathcal{L}_m(s)=2|\mathcal{L}_m(s)|-1$.

    When $s=2t-1$ is odd, it's obvious that
    \begin{equation*}
        |\mathcal{L}_m(s)|=\sum_{i=1}^m |P_k|
        \sum_{i=1}^{m-1} (2+4+\cdots+(2t-2))+(1+3+\cdots+(2t-3))
        =(t-1)(mt-1).
    \end{equation*}
    On the other hand, $\max\mathcal{L}_m(s)$ equals
    \begin{equation*}
        m(t-1)(2t-1)-t=2|\mathcal{L}_m(s)|-1-(t-1)(m-1)<2|\mathcal{L}_m(s)|-1,
    \end{equation*}
    a contradiction.

    When $s=2t-2$ is even, we can similarly prove that $\max\mathcal{L}_m(s)<2|\mathcal{L}_m(s)|-1$. A contradiction.
\end{proof}

\section{Generalized-$\beta$-sets}

Now we generalize the concept of the $\beta$-set of an $(s,ms-1,ms+1)$-core partition to the generalized-$\beta$-set. In the following sections, we will use this more general concept instead of the original $\beta$-sets studied in previous papers (see \cite{James,NS2, Olsson, Xiong1}). 
To better formulate the concept of a generalized-$\beta$-set, we introduce the following notation. For nonempty finite set $S\subseteq[\lceil(s-1)/2\rceil ms-1]=\{ x\in \mathbb{N}: 1\leq x \leq  \lceil(s-1)/2\rceil ms-1 \}$, define $$t(S):=\min\{i \geq 1: S\subseteq[0,ims-1]\}.$$ For $1\leq i\leq \lceil(s-1)/2\rceil m$, define
\[
\mathcal{B}_i(S):=S\cap [(i-1)s, is-1];  
\]
\[
a_i(S):=|S\cap [(i-1)s, is-1]| ;
\]
\[
n_i(S):=\max~ \{x~ \text{mod}~s: x\in  \mathcal{B}_i(S) \} ~~\text{(we set $n_i(S)=0$ if $\mathcal{B}_i(S)=\emptyset$)}.
\]
For $1\leq k\leq \lceil(s-1)/2\rceil$, define
\[
A_k(S):=\sum_{j=1}^{m}a_{(k-1)m+j}(S).
\]
When there is only one set $S$, we write $t,\mathcal{B}_i,a_i,n_i$ instead of $t(S),\mathcal{B}_i(S),a_i(S),n_i(S)$ for simplicity. When we deal with two sets $S,S'$, the same is for $S$ and we use $t',\mathcal{B}_i',a_i',n_i'$ instead of $t(S'),\mathcal{B}_i(S'),a_i(S'),n_i(S')$. 
The notation above is inspired by Example \ref{ex:1.2}. Let~$S$ be the $\beta$-set of a maximal $(s,ms-1,ms+1)$-core partition. Observing Example \ref{ex:1.2}, we find that $S$ can be divided into several parts, namely $\mathcal{B}_i=S\cap[(i-1)s,is-1]$, each of which consists of consecutive integers. The shape of these consecutive integers is determined by the cardinality of $\mathcal{B}_i$, i.e. $a_i$ and the maximal element, which can be represented by its modulo mod $s$, i.e.~$n_i$. The notation $t$ is just the number of non-zero rows that $S$ has in $\mathcal{L}_m(s)$.

\begin{example}
Let $m=3$, $s=5$ and $S=\mathcal{L}_3(5)$. Then $t=2$. We have (see Figure \ref{L_3(5)-2})
\[
\mathcal{B}_1=\{1,2,3,4\},a_1=4,n_1=4;
\]
\[
\mathcal{B}_2=\{6,7,8,9\},a_2=4,n_2=4;
\]
\[
\mathcal{B}_3=\{11,12,13\},a_3=3,n_3=3;
\]
\[
\mathcal{B}_4=\{17,18\},a_4=2,n_4=3;
\]
\[
\mathcal{B}_5=\{22,23\},a_5=2,n_5=3;
\]
\[
\mathcal{B}_6=\{27\},a_6=1,n_6=2;
\]
\[
A_1=a_1+a_2+a_3=11,A_2=a_4+a_5+a_6=5.
\]
\end{example}


Next, we give the definition of generalized-$\beta$-sets.
\begin{definition}[generalized-$\beta$-set] \label{def:4.2}
Let $S\subseteq  \mathcal {L}_m(s)$ be a nonempty set. Then, $S$ is called a \emph{generalized-$\beta$-set}, if it satisfies the following conditions:
\begin{enumerate}
\item[(1)]  If $1\leq i<\lceil(s-1)/2\rceil m$ and $a_i=0$, then $a_{i+1}=0$;
    \item[(2)]  If $1\leq i < tm$ and $m\nmid i$, then $a_i\ge a_{i+1}$;
    \item[(3)]  If $1\leq i\leq (t-1)m$ and $a_{i+m}>0$, then $a_{i+m}\leq a_{i}-2$;
    \item[(4)]  If $1\leq i < tm$, $m\nmid i$ and $a_{i+1}>0$, then $n_i \ge n_{i+1}$;
    \item[(5)]  If $1\leq i < tm$, $m\nmid i$ and $a_i=a_{i+1}>0$, then $n_i=n_{i+1}$;
    \item[(6)] $a_{(t-1)m}\ge a_{(t-1)m+1}$.
\end{enumerate}
\end{definition}

\begin{example}
When $m=3$ and $s=5$, $\mathcal{L}_3(5)$ is a generalized-$\beta$-set. In fact, it is obvious that (1) holds. Since $(a_1,a_2,a_3,a_4,a_5,a_6)=(4,4,3,2,2,1)$ is non-increasing, (2) and (6) also hold. Since $a_4=a_1-2$, $a_5=a_2-2$ and $a_6=a_3-2$, (3) is true. Since $(n_1,n_2,n_3,n_4,n_5,n_6)=(4,4,3,3,3,2)$ is non-increasing, (4) is true. If $a_i=a_{i+1}$, then $i=1$ or $i=4$. Since $n_1=n_2$ and $n_4=n_5$, (5) is also true.
\end{example}


The following lemma is the reason why we use the term \emph{generalized-$\beta$-set}.

\begin{lemma}\label{if}
Let $\lambda$ be an $(s,ms-1,ms+1)$-core partition. Then the $\beta$-set $S:=\beta(\lambda)$ of $\la$ is a generalized-$\beta$-set.
\end{lemma}

\begin{proof}

\par Recall that $\mathcal{B}_i=S\cap [(i-1)s, is-1]$ for $1 \leq i \leq tm$. For $i\ge 2$, assume that $\mathcal{B}_i\neq \emptyset$. Set $\mathcal{B}_i=\{x_1,x_2,\cdots,x_k\}$. By Theorem \ref{L_m(s)}, $x_i-s\in \mathcal{B}_{i-1}$. Therefore, 
\[
a_{i-1}= |\mathcal{B}_{i-1}|\ge k=|\mathcal{B}_i|=a_i
\] and
\[
n_{i-1}=\max \mathcal{B}_{i-1}(\text{mod}~s)\ge \max (\mathcal{B}_i-s)(\text{mod}~s)=\max \mathcal{B}_i(\text{mod}~s)=n_i.
\]Therefore, Conditions (1), (2), (4) and (6) in Definition \ref{def:4.2} are satisfied. If $a_{i-1}=a_i$, then $\mathcal{B}_{i-1}=\{x_1-s,\cdots,x_k-s\}$, thus $n_{i-1}=n_i$. Therefore, Condition (5) in Definition \ref{def:4.2} is satisfied.
For $m+1\leq i\leq tm$, notice that $\{x_1-ms-1,x_2-ms-1,\cdots,x_k-ms-1,x_k-ms,x_k-ms+1\} \subseteq   \mathcal{B}_{i-m}$. Therefore,
\[
|\mathcal{B}_{i-m}| \ge k+2 = |\mathcal{B}_i|+2.
\] 
Thus Condition (3) in Definition \ref{def:4.2} is also satisfied.
\end{proof}

\begin{remark}
A generalized-$\beta$-set is not necessarily the $\beta$-set of an $(s,ms-1,ms+1)$-core partition. Let $s=6$ and $m=3$. See Figure \ref{fig:nobe} for example, $S=\{1,2,3,4,7,8,9,10,13,14,\\ 15,16,21,22,27,33\}$ is not the $\beta$-set of a $(6,17,19)$-core partition, since $22-17=5\not\in S$, which violates Condition (2) in Theorem \ref{L_m(s)}. Observe that $(a_1,a_2,a_3,a_4,a_5,a_6)=(4,4,4,2,1,1)$ and $(n_1,n_2,n_3,n_4,n_5,n_6)=(4,4,4,4,3,3)$. We verify that $S$ is indeed a generalized-$\beta$-set. Obviously (1) holds. Both sequences are non-increasing, thus (2)(4)(6) hold. Since $a_1 - a_4 = 2$ and $a_2 - a_5 = a_3 - a_6 = 3$, (3) holds. Since $n_1 = n_2 = n_3$ and $n_5 = n_6$, (5) holds.
\end{remark}

\begin{figure}[h!]
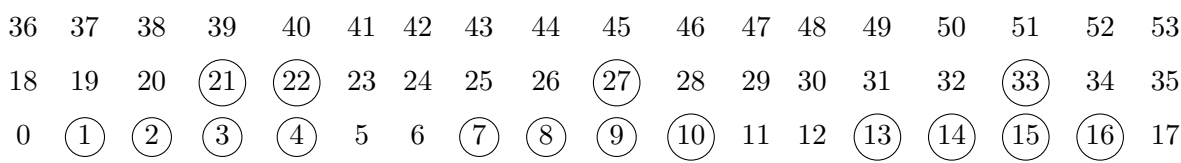

    \centering
    \[ \small
    \begin{array}{cccccccccccccccccc}
    36 & 37 & 38 & 39 & 40 & 41 & 42 & 43 & 44 & 45 & 46 & 47 & 48 & 49 & 50 & 51 & 52 & 53 \\
    18 & 19 & 20 & \encircle{21} & \encircle{22} & 23 & 24 & 25 & 26 & \encircle{27} & 28 & 29 & 30 & 31 & 32 & \encircle{33} & 34 & 35\\
    0 & \encircle{1} & \encircle{2} & \encircle{3} & \encircle{4} & 5 & 6 & \encircle{7} & \encircle{8} & \encircle{9} & \encircle{10} & 11 & 12 & \encircle{13} & \encircle{14} & \encircle{15} & \encircle{16} & 17\\
    
    \end{array}
    \]
    \caption{A generalized-$\beta$-set which is not a $\beta$-set.}
    \label{fig:nobe}
    \end{figure}


\par The following proposition gives an upper bound of $n_i$.

\begin{proposition} \label{upper bound}
Let $S\subseteq  \mathcal {L}_m(s)$ be a generalized-$\beta$-set. For $0 \leq i \leq t-1 $ and $1\leq k \leq m$, we have
\[
n_{im+k}\leq
\begin{cases}
s-i-2 \ \ if \ a_{im+k}=a_{(i+1)m};\\
s-i-1 \ \ if \ a_{im+k}>a_{(i+1)m}.
\end{cases}
\]
\end{proposition}

\begin{proof}
Since $S\subseteq\mathcal{L}_m(s)$, we have
\[
n_{im+k}\leq
\begin{cases}
s-i-2 \ \ \text{if} \ k=m;\\
s-i-1 \ \ \text{if} \ 1\leq k<m.
\end{cases}
\]Then by Condition (5) in Definition \ref{def:4.2}, we obtain the result.
\end{proof}




The next proposition derives several inequalities for $A_i$.

\begin{proposition} \label{<2m}
Let $S\subseteq  \mathcal {L}_m(s)$ be a generalized-$\beta$-set. Then,
\begin{enumerate}
    \item[(1)] $A_i-A_{i+1}\ge 2m$ for $1\leq i\leq t-2$;
    \item[(2)] If $A_{t-1}-A_t\leq 2m-1$, then there exists $1\leq p<m$, such that $a_{(t-1)m}=a_{(t-1)m+1}=\cdots=a_{(t-1)m+p}=1$ and $a_{(t-1)m+p+1}=\cdots=a_{tm}=0$.
\end{enumerate}

\end{proposition}

\begin{proof}
By the definition of $t$, we have $a_{(t-1)m+1}>0$. By Condition (1) in Definition \ref{def:4.2}, we have $a_i>0$ for $1\leq i\leq (t-1)m+1$. Then Conclusion (1) can be directly derived by Condition (3) in Definition \ref{def:4.2}. Next, we prove Conclusion (2), assuming that $A_{t-1}-A_t\leq 2m-1$. By Condition (6) in Definition \ref{def:4.2} we have $a_{(t-1)m}\ge a_{(t-1)m+1}>0$. If $a_{(t-1)m}\ge 2$, then $a_{(t-2)m+k}\ge a_{(t-1)m}\ge 2$ for $1\leq k\leq m$ by Condition (2) in Definition \ref{def:4.2}. Therefore, $a_{(t-2)m+k}-a_{(t-1)m+k}\ge 2$ for $1\leq k\leq m$. This contradicts with $A_{t-1}-A_t<2m$. Therefore, $a_{(t-1)m}=1$. Note that $0<a_{(t-1)m+1}\leq a_{(t-1)m}$ and $a_{tm}<a_{(t-1)m}$, we derive Conclusion (2) in this proposition.
\end{proof}

\begin{example}
(1) Let $m=3$, $s=5$, and $S=\mathcal{L}_3(5)$. Then we have $A_1=11$ and $A_2=5=11-2\times 3=A_1-2m$. 

(2) Let $m=3$, $s=6$, and $S=\mathcal{L}_3(6)$. Then we have $A_1=14$, $A_2=8=14-2\times 3=A_1-2m$ and $A_3=2=8-2\times 3=A_2-2m$.
\end{example}






\par Let $\mathbb{A}_m(s) $ be the set of all $\beta$-sets of $(s,ms-1,ms+1)$-core partitions and $\mathbb{B}_m(s) $ be the set of all generalized-$\beta$-sets $S\subseteq  \mathcal {L}_m(s)$. By Lemma \ref{if}, we know that
\begin{equation*} \max_{S'\in\mathbb{A}_m(s)}f(S')\leq\max_{S''\in\mathbb{B}_m(s)}f(S'').
\end{equation*}
We aim to prove Nath-Sellers' conjecture by showing that
\begin{equation*}
\max_{S'\in\mathbb{A}_m(s)}f(S')=\max_{S''\in\mathbb{B}_m(s)}f(S'')
\end{equation*}
and
\begin{equation*}
    {\arg}\max_{S'\in\mathbb{A}_m(s)}f(S')={\arg}\max_{S''\in\mathbb{B}_m(s)}f(S''),
\end{equation*}
which will be proved in the following sections while we are searching for all generalized-$\beta$-sets $S$ that maximize $f(S)$. Here and in the rest of the paper, for a function $g: D \to \mathbb{R}$,
\begin{equation*}
    {\arg}\max_{x \in D} g(x) := \{t \in D : g(x) = \max_{t \in D} g(x) \}.
\end{equation*}
\par We use generalized-$\beta$-sets instead of $\beta$-sets since the original definition of the $\beta$-set is too tight for the adjustments in Sections $5$, $6$ and $7$. We can easily make a {generalized-$\beta$-set} remain a {generalized-$\beta$-set} after each adjustment in our following proofs, while a $\beta$-set may not remain a $\beta$-set after the adjustments.

\section{Preliminaries for the proofs}


In the next three sections, we will introduce several adjustments for a generalized-$\beta$-set $S$ which doesn't maximize $f(S)$. After the adjustments, we will get a new set $S'$, such that $f(S)<f(S')$. We wish to prove that $S'$ is a generalized-$\beta$-set. The following lemma will be repeatedly used to verify Condition (3) in Definition \ref{def:4.2} for $S'$.

\begin{lemma} \label{One Line}
Let $x_1,x_2,\cdots,x_n$ and $y_1,y_2,\cdots,y_n$ be two non-increasing integer sequences such that $|x_n-x_1|\leq 1$ and $|y_n-y_1|\leq 1$. Assume that
\[
\sum_{i=1}^{n} x_i \leq \sum_{i=1}^{n} y_i - 2n,
\]
then $x_i\leq y_i-2$ for $1\leq i\leq n$.
\end{lemma}

\begin{proof}

\par We prove it by contradiction. Assume that there exists $1\leq i\leq n$, such that $x_i>y_i-2$, then $x_i\ge y_i-1$. For $j>i$, we have $x_j\ge x_i-1\ge y_i-2\ge y_j-2$. For $j<i$, we have $x_j\ge x_i\ge y_i-1\ge y_j-2$. Therefore, $\sum_{j=1}^{n} x_j>\sum_{j=1}^{n} y_j - 2n$. A contradiction.
\end{proof}



The following lemma shifts our focus from all generalized-$\beta$-sets to the ones with some specific structure.

\begin{lemma}\label{initial}
Let $S$ be a generalized-$\beta$-set that maximizes $f(S)$. Then 
\begin{align}\label{eq: 5.1}
   \mathcal{B}_i=S\cap[(i-1)s,is-1] 
\end{align} consists of consecutive integers for $1\leq i\leq tm$. Furthermore, for $0\leq i\leq t-1$ and $1\leq k\leq m$, we have 
    \begin{align}\label{eq: 5.2}
n_{im+k}=
\begin{cases}
s-i-2 \ \ \text{if} \ a_{im+k}=a_{(i+1)m};\\
s-i-1 \ \ \text{if} \ a_{im+k}>a_{(i+1)m}.
\end{cases}
\end{align}
\end{lemma}

\begin{proof}
\par 
First, we construct a generalized-$\beta$-set $S'\subseteq  \mathcal {L}_m(s)$. For $0\leq i\leq t-1$ and $1\leq k\leq m$, if $a_{im+k}=a_{(i+1)m}$, let $\mathcal{B}'_{im+k}=\{(im+k)s-i-a_{im+k}-1,(im+k)s-i-a_{im+k},\cdots,(im+k)s-i-2\}$. Otherwise, let $\mathcal{B}'_{im+k}=\{(im+k)s-i-a_{im+k},(im+k)s-i-a_{im+k}+1,\cdots,(im+k)s-i-1\}$. Then $\mathcal{B}'_{im+k}$ consists of $a_{im+k}=|S\cap [(im+k-1)s, (im+k)s-1]|$ consecutive numbers, and $n_i'=\max~ \{x~ \text{mod}~s: x\in  \mathcal{B}_i(S') \} $ satisfies the equality in Proposition \ref{upper bound}.
\par Therefore, Conditions (1)(2)(3)(6) in Definition \ref{def:4.2} hold for $S'$ since $a_i'=a_i$ for $1\leq i\leq tm$. Conditions (4)(5) in Definition \ref{def:4.2} are also true since the equality in Proposition \ref{upper bound} is achieved. Thus $S'$ is also a generalized-$\beta$-set and $|S|=|S'|$.
\par For $1\leq i\leq tm$, $\sum_{x\in\mathcal{B}_i}x\leq\sum_{x\in\mathcal{B}'_i}x$. Therefore,
\begin{equation*}
    f(S)=\sum_{i=1}^{tm}\sum_{x\in\mathcal{B}_i}x-\frac{|S|(|S|-1)}{2} 
    \leq\sum_{i=1}^{tm}\sum_{x\in\mathcal{B}'_i}x-\frac{|S'|(|S'|-1)}{2}
    = f(S').
\end{equation*}
If $S \neq S'$, the equality above cannot be achieved, which means that $f(S)<f(S')$. However, at the beginning we assume that $S$ is a generalized-$\beta$-set that maximizes $f(S)$, which means that $f(S)\geq f(S')$,  a {contradiction}. Therefore, $S=S'$ and thus for $0\leq i\leq t-1$ and $1\leq k\leq m$, $\mathcal{B}_{im+k}$ consists of consecutive integers and
\[
n_{im+k}=
\begin{cases}
s-i-2 \ \ if \ a_{im+k}=a_{(i+1)m};\\
s-i-1 \ \ if \ a_{im+k}>a_{(i+1)m}.
\end{cases}
\]
\end{proof}

\begin{definition}\label{def:5.3}
Recall that $\mathbb{A}_m(s) $ is the set of all $\beta$-sets of $(s,ms-1,ms+1)$-core partitions and $\mathbb{B}_m(s) $ is the set of all generalized-$\beta$-sets $S\subseteq  \mathcal {L}_m(s)$. Let $\mathbb{C}_m(s)$ be the set of all generalized-$\beta$-sets $S\subseteq\mathcal{L}_m(s)$ such that $S$ satisfies Conditions \eqref{eq: 5.1} and \eqref{eq: 5.2} in Lemma \ref{initial}. 
Let $\mathbb{D}_m(s)$ be the set of all nonempty sets $S\subseteq\left[\lceil (s-1)/2 \rceil ms-1 \right]$ such that the nonempty set $\mathcal{B}_i$ consists of $a_i$ consecutive integers for $1\leq i\leq \lceil (s-1)/2 \rceil m$. 
\end{definition}

The following lemma is obvious and the proof is omitted.

\begin{lemma}
For the sets $\mathbb{C}_m(s)$ and $\mathbb{D}_m(s)$ defined in Definition \ref{def:5.3}, we have 
\begin{enumerate}
    \item[(1)] $\mathbb{C}_m(s)\subseteq\mathbb{D}_m(s)$;
    \item[(2)] $\mathbb{C}_m(s)\subseteq\mathbb{B}_m(s)$.
\end{enumerate}
    
\end{lemma}

We use Figure \ref{fig:abcd} to visualize that $\mathbb{A}_m(s) \subseteq \mathbb{B}_m(s), \mathbb{C}_m(s) \subseteq \mathbb{B}_m(s), \mathbb{C}_m(s) \subseteq \mathbb{D}_m(s)$.

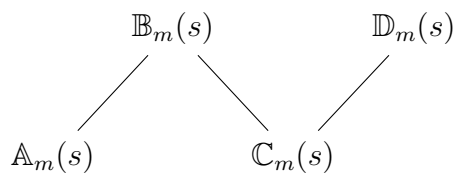
\begin{figure}[h!]
    \centering
    \begin{tikzpicture}[node distance=.2cm]
    \node(B) {$\mathbb{B}_m(s)$};
    \node(A) [below left=1cm and .2cm of B] {$\mathbb{A}_m(s)$};
    \node(C) [below right=1cm and .2cm of B] {$\mathbb{C}_m(s)$};
    \node(D) [above right=1cm and .2cm of C] {$\mathbb{D}_m(s)$};
    \draw (A) -- (B);
    \draw (B) -- (C);
    \draw (C) -- (D);
    \end{tikzpicture}
    \caption{Relations between $\mathbb{A}_m(s), \mathbb{B}_m(s), \mathbb{C}_m(s),$ and $ \mathbb{D}_m(s)$.}
    \label{fig:abcd}
\end{figure}

\begin{remark}
By Lemma \ref{initial} it is easy to see that 
\begin{equation*}
    {\arg}\max_{S'\in\mathbb{B}_m(s)}f(S')={\arg}\max_{S''\in\mathbb{C}_m(s)}f(S'').
\end{equation*}

\end{remark}


The following lemma gives a characterization of $\mathbb{D}_m(s)$.

\begin{lemma}
    There is a bijection $F: \mathbb{D}_m(s)\to V$, 
    where $V\subseteq\Z^{2\lceil (s-1)/2 \rceil m}$, and for $Z=\left(x_1,\ldots,x_{\lceil (s-1)/2 \rceil m};y_1,\ldots,y_{\lceil (s-1)/2 \rceil m}\right)\in\Z^{2\lceil (s-1)/2 \rceil m}$, $Z\in V$ if and only if $Z$ satisfies the following conditions:
    \begin{enumerate}
        \item[(1)] For $1\leq i\leq \lceil (s-1)/2 \rceil m$, $0\leq x_i\leq s-1$ and $0\leq y_i\leq s$;
        \item[(2)] For $1\leq i\leq \lceil (s-1)/2 \rceil m$, $x_i\ge y_i-1$.
    \end{enumerate}
\end{lemma}

\begin{proof}
Let $N=\lceil (s-1)/2 \rceil m$. For $S\in\mathbb{D}_m(s) $, let
\[
F(S)=\left(n_1(S),n_2(S),\cd,n_N(S);a_1(S),a_2(S),\cd,a_N(S)\right).
\]Then for $1\leq i\leq N$, by definition we have $0\leq n_i(S)\leq s-1, ~0\leq a_i(S)\leq s$ and $a_i(S)-1\leq n_i(S)$. Therefore, $f(S)\in V$ and $F$ are well-defined.
\par We first prove that $F$ is injective. Assume that $S,S'\in\mathbb{D}_m(s)$ such that $F(S)=F(S')$, then $(n_i,a_i)=(n_i',a_i')$ for $1\leq i\leq N$. Then $\mathcal{B}_i=\{(i-1)s+n_i-a_i+1,(i-1)s+n_i-a_i+2,\cdots,(i-1)s+n_i\}=\{(i-1)s+n_i'-a_i'+1,(i-1)s+n_i'-a_i'+2,\cdots,(i-1)s+n_i'\}=\mathcal{B}_i'$. Therefore, we have $S=S'$.
\par Next, we prove that $F$ is surjective. For any $Z=\left(x_1,\cdots,x_N;y_1,\cdots,y_N\right)\in V$, let $\mathcal{B}_i=\{(i-1)s+n_i-a_i+1,(i-1)s+n_i-a_i+2,\cdots,(i-1)s+n_i\}$. Then $\mathcal{B}_i\subseteq [(i-1)s,is-1]$ consists of consecutive integers. Therefore, we have $S=\cup_{i=1}^{N}\mathcal{B}_i\subseteq\mathbb{D}_m(s) $ and $F(S)=Z$.
\par From the discussion above, we obtain that $F$ is bijective.
\end{proof}

Next, we show how to calculate $f(S)$ for $S\in\mathbb{D}_m(s)$.

\begin{lemma}\label{lem:5.7}
Let $S\in\mathbb{D}_m(s) $. Then
\begin{equation*}
    \sum_{x\in\mathcal{B}_i}x=a_i(i-1)s+n_ia_i-\frac{a_i(a_i-1)}{2}
\end{equation*}and
\begin{equation*}
    f(S)=\sum_{i=1}^{tm}\left(a_i(i-1)s+n_ia_i-\frac{a_i(a_i-1)}{2}\right)-\frac{\sum_{i=1}^{tm}a_i\left(\sum_{i=1}^{tm}a_i-1\right)}{2}.
\end{equation*}
\end{lemma}

\begin{proof}
Since $f(S)=\sum_{i=1}^{tm}\sum_{x\in\mathcal{B}_i}x-|S|(|S|-1)/2$, we only need to prove the first equation. In fact, since $\mathcal{B}_i=\{(i-1)s+n_i-a_i+1,(i-1)s+n_i-a_i+2,\cd,(i-1)s+n_i\}$, we have
\begin{equation*}    \sum_{x\in\mathcal{B}_i}x=\sum_{j=0}^{a_i-1}\left((i-1)s+n_i-j\right)=a_i\left((i-1)s+n_i\right)-\frac{a_i(a_i-1)}{2}.
\end{equation*}
\end{proof}

Lemma \ref{lem:5.7} implies the following result.

\begin{corollary}\label{compare}
Let $S,S'\subseteq\mathbb{D}_m(s)$ and $t=t'$. Assume that there exist positive integers $i<j$, such that the following three conditions hold:
\begin{enumerate}
    \item[(1)] $(n_i',a_i')=(n_i-1,a_i-1)$ or $(n_i,a_i-1)$; \item[(2)] $(n_j',a_j')=(n_j+1,a_j+1)$ or $(n_i,a_i+1)$;
    \item[(3)] For all $1\leq k\leq tm$, $k\neq i,j$, we have $(n_k,a_k)=(n_k',a_k')$.
\end{enumerate}
Then $f(S)<f(S')$.
\end{corollary}

\begin{proof}
If $(n_i',a_i')=(n_i-1,a_i-1)$  and $(n_j',a_j')=(n_j+1,a_j+1)$, then there exist some $x_0\in\mathcal{B}_i$ and $y_0\in\left(\mathcal{L}_m(s)\cap[(j-1)s,js-1]\right)\backslash\mathcal{B}_j $, such that $\mathcal{B}_i'=\mathcal{B}_i\backslash\{x_0\}$ and $\mathcal{B}_j'=\mathcal{B}_j\cup\{y_0\}$. Thus, 
\begin{equation*}
    f(S')-f(S)=\left(\sum_{x\in\mathcal{B}_i'} x-\sum_{x\in\mathcal{B}_i} x\right)+\left(\sum_{y\in\mathcal{B}_j'} y-\sum_{y\in\mathcal{B}_j} y\right)=y_0-x_0>0.
\end{equation*}
For other cases of $(n_i',a_i')$ and $(n_j',a_j')$, the proof is similar.
\end{proof}

\section{Adjustments in one row}

\par 
The aim of this section is to show that if $S$ maximizes $f(S)$, then $a_{(i-1)m+1}-a_{im}\leq 2$ for $1\leq i\leq t$, just as Example 1.2 shows.
The following lemma gives some properties of a generalized-$\beta$-set $S$ that maximizes $f(S)$.






\begin{lemma} \label{row}
Assume that a generalized-$\beta$-set $S$ maximizes $f(S)$. Then we have the following conclusions:
\begin{enumerate}
    \item[(1)] If $A_{t-1}-A_t\ge 2m$, then
\[
a_{(i-1)m+1}-a_{im}\leq 1
\]
for $1\leq i\leq t$.
\item[(2)] If $A_{t-1}-A_t\leq 2m-1$, by Proposition \ref{<2m}, there exists $1\leq p<m$, such that $a_{(t-1)m}=a_{(t-1)m+1}=\cdots=a_{(t-1)m+p}=1$ and $a_{(t-1)m+p+1}=\cdots=a_{tm}=0$. Then
\[
a_{(i-1)m+1}-a_{(i-1)m+p}\leq 1
\]
and
\[
a_{(i-1)m+p+1}-a_{im}\leq 1
\]
for $1\leq i\leq t$.
\end{enumerate}
\end{lemma}

\begin{remark}
Note that the inequality $a_{i+m}\leq a_{i}-2$ may be violated when $a_{i+m}=0$. This needs to be discussed in the following proof.
\end{remark}

\begin{proof}

\par By Lemma \ref{initial}, $S\in\mathbb{C}_m(s)$.  
{We prove Conclusion (2) here. The proof of Conclusion (1) is similar and omitted. Assume that Conclusion (2) is not true}. {Then either there exists some~${0\leq i\leq t-2}$, such that ${a_{im+p+1}-a_{(i+1)m}\ge 2}$; or there exists some~${0\leq j\leq t-2}$, such that ${a_{jm+1}-a_{jm+p}\ge 2}$}. These two cases are similar. Without loss of generality, we consider the first case, i.e., there exists some~$0\leq i\leq t-2$, such that $a_{im+p+1}-a_{(i+1)m}\ge 2$.

Assume that $a_{im+p+1}=a_{im+p+2}=\cdots=a_{im+p+u}>a_{im+p+u+1}$ and $a_{(i+1)m}=a_{(i+1)m-1} \\ =\cdots=a_{(i+1)m-v}<a_{(i+1)m-v-1}$ for some positive integers $u$ and $v$. Then $u+v<m-p$. We will do adjustments to $S\cap [(im+p)s,(i+1)ms-1]$, which can be divided into the following three cases.


    \begin{enumerate}
    
    \item[(1)] {${v=0}$ and ${a_{(i+1)m-1}\ge a_{(i+1)m}+2}$}. Let $S'\in\mathbb{D}_m(s)$ satisfy the following conditions:
\begin{enumerate}
    \item $t'=t$;
    \item For $1\leq k\leq tm$, $\left(n_k',a_k'\right)=\left(n_k,a_k\right)$ if $k\neq (i+1)m$ and $k\neq (i+1)m-1$;
    \item If ${a_{(i+1)m-1} = a_{(i+1)m}+2}$, then \[
    \left(n_{(i+1)m-1}',a_{(i+1)m-1}'\right)=\left(n_{(i+1)m-1}-1,a_{(i+1)m-1}  -1\right);
    \]
    If ${a_{(i+1)m-1}\ge a_{(i+1)m}+3}$, then
    \[
    \left(n_{(i+1)m-1}',a_{(i+1)m-1}'\right)=\left(n_{(i+1)m-1},a_{(i+1)m-1}  -1\right);
    \]
    \item $\left(n_{(i+1)m}',a_{(i+1)m}'\right)=\left(n_{(i+1)m},a_{(i+1)m}+1\right)$.
\end{enumerate}It's easy to check that $S'$ is well-defined and $|S|=|S'|$. Then by Corollary \ref{compare}, we have $f(S')>f(S)$.
    
    \item[(2)] {${v=0}$ and ${a_{(i+1)m-1}=a_{(i+1)m}+1}$}. Assume that $a_{(i+1)m-1}=\cdots=a_{(i+1)m-w}<a_{(i+1)m-w-1}$. Then $w + u < m - p$ since $a_{im+p+1}-a_{(i+1)m}\ge 2$. Let $S'\in\mathbb{D}_m(s)$ satisfy the following conditions:
\begin{enumerate}
    \item $t'=t$;
    \item For $1\leq k\leq tm$, $(n_k',a_k')=(n_k,a_k)$ if $k\not\in\{im+p+u,(i+1)m-w,(i+1)m-w+1,\cd,(i+1)m\}$;
    \item If $a_{im+p+u} = a_{(i+1)m}+2$, then $\left(n_{im+p+u}',a_{im+p+u}'\right)=\left(n_{im+p+u}-1,a_{im+p+u}-1\right)$; If $a_{im+p+u} \ge a_{(i+1)m}+3$, then $\left(n_{im+p+u}',a_{im+p+u}'\right)=\left(n_{im+p+u},a_{im+p+u}-1\right)$;
    \item For $1\leq r\leq w$, $\left(n_{(i+1)m-r}',a_{(i+1)m-r}'\right)=\left(n_{(i+1)m-r}-1,a_{(i+1)m-r}\right)$;
    \item $\left(n_{(i+1)m}',a_{(i+1)m}'\right)=\left(n_{(i+1)m},a_{(i+1)m}+1\right)$.
\end{enumerate}It's easy to check that $S'$ is well-defined and $|S|=|S'|$. Let  $D(i)$ denote $\sum_{x\in\mathcal{B}_i'}x-\sum_{x\in\mathcal{B}_i}x$, then
    
    \begin{align*}
    f(S')-f(S)
     & = D_{im+p+u} + D_{(i+1)m} + \sum_{r=1}^w D_{(i+1)m-r} \\
     &\ge ws - w (a_{(i+1)m}+1) \\
    & =  w(s-1-a_{(i+1)m}) \\
    & > 0.
    \end{align*}
    Thus $f(S')>f(S)$. {Notice that here $S'$ may not be a generalized-$\beta$-set}.
    \item[(3)]  {${v>0}$}. The construction is similar and omitted. 
        
    \end{enumerate}

\par The adjustments when there exists some~${0\leq j\leq t-2}$, such that ${a_{jm+1}-a_{jm+p}\ge 2}$ are similar. 
The process above can be repeated in $S'$ if there exists some $0\leq i\leq t-2$, such that $a_{im+p+1}'-a_{(i+1)m}'\ge 2$ or there exists some $0\leq j\leq t-2$, such that $a_{jm+1}'-a_{jm+p}'\ge 2$. Notice that for any $T\in\mathbb{D}_m(s)$, $f(T)\leq \lceil (s-1)/2 \rceil^2 m^2 s^2$. Also
\[
f(S') \ge f(S) + 1,
\]
thus the adjustments will stop after finite steps, when we get a final set $S''$, such that $a_{(i-1)m+1}''-a_{(i-1)m+p}''\leq 1$ and $a_{(i-1)m+p+1}''-a_{im}''\leq 1$ for $1\leq i<t$. For $S$, set
\[
U_i=\sum_{j=1}^{p} a_{(i-1)m+j}
\]
and
\[
V_i=\sum_{j=p+1}^{m} a_{(i-1)m+j}
\]
for $1\leq i\leq t$. The analog notations $U_i''$ and $V_i''$ are for $S''$. By the discussion above, for $1\leq i\leq t$, we have $U_i=U_i''$ and $V_i=V_i''$. Since $S$ is a generalized-$\beta$-set, we obtain
\[
U_i-U_{i+1}\ge2p
\]for $1\leq i\leq t-1$ and 
\[
V_i-V_{i+1}\ge2(m-p)
\]for $1\leq i\leq t-2$.
From Lemma \ref{One Line}, $S''$ satisfies Condition (3) in Definition \ref{def:4.2}. It is easy to check that $S''$ satisfies Conditions (1)(2)(6) in Definition \ref{def:4.2}. Condition (2) in Lemma \ref{initial} hold for $S'$ in every case of adjustment. So 
$S''$ satisfies Conditions (4)(5) in Definition \ref{def:4.2}. Thus $S''$ is also a generalized-$\beta$-set and ${f(S)<f(S'')}$, which contradict the assumption that $S$ is a generalized-$\beta$-set that maximizes $f(S)$.
\end{proof}

With this lemma, we can control the shape of a certain row of a generalized-$\beta$-set $S$ that maximizes $f(S)$.

\begin{definition}
    Let $\mathbb{E}_m(s)$ be the set of all $S\in\mathbb{C}_m(s)$ such that $S$ satisfies Conclusions (1) and (2) in Lemma \ref{row}.
\end{definition}

By Lemma \ref{row}, we obtain the following corollary.

\begin{corollary}
We have $\mathbb{E}_m(s)\subseteq\mathbb{C}_m(s)\subseteq\mathbb{B}_m(s)$ and
\begin{equation*}
    {\arg}\max_{S'\in\mathbb{B}_m(s)}f(S')={\arg}\max_{S''\in\mathbb{C}_m(s)}f(S'')={\arg}\max_{S'''\in\mathbb{E}_m(s)}f(S''').
\end{equation*}
\end{corollary}

Next, we give an example of the adjustments in Lemma \ref{row}.

\begin{example}
\par Let $m=3$ and $s=6$. Figure \ref{RowBefore} shows the generalized-$\beta$-set $S$ before the adjustment. Figure \ref{RowAfter} shows the generalized-$\beta$-set $S'$ after the  adjustment. 

Notice that $S'=S\setminus\{11\}\cup\{13\}$. We color $S \setminus S'$ blue and $S' \setminus S$ red. The adjustment is on the first row and follows case (1) in the proof, where $i=v=0$ and $a_{(i+1)m-1}=a_{(i+1)m}+2$. We can see that $(n_2,a_2)=(5,5),(n_3,a_3)=(4,3),(n_2',a_2')=(4,4),(n_3',a_3')=(4,4)$. Indeed, $(n_2',a_2')=(n_2-1,a_2-1)$ and $(n_3',a_3')=(n_3,a_3+1)$. We can verify that $f(S')-f(S)=68-66=13-11>0.$

    \begin{figure}[h!]
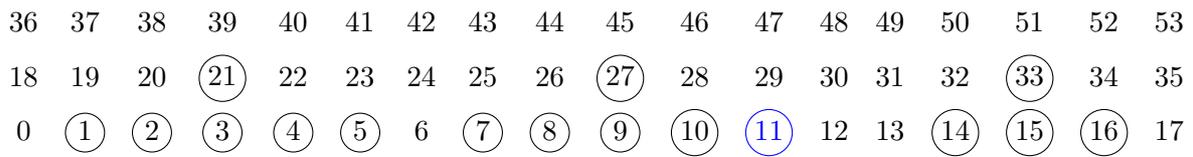

    \centering
    \[ \small
    \begin{array}{cccccccccccccccccc}
    36 & 37 & 38 & 39 & 40 & 41 & 42 & 43 & 44 & 45 & 46 & 47 & 48 & 49 & 50 & 51 & 52 & 53 \\
    18 & 19 & 20 & \encircle{21} & 22 & 23 & 24 & 25 & 26 & \encircle{27} & 28 & 29 & 30 & 31 & 32 & \encircle{33} & 34 & 35\\
    0 & \encircle{1} & \encircle{2} & \encircle{3} & \encircle{4} & \encircle{5} & 6 & \encircle{7} & \encircle{8} & \encircle{9} & \encircle{10} & \enblue{11} & 12 & 13 & \encircle{14} & \encircle{15} & \encircle{16} & 17\\
    
    \end{array}
    \]
    \caption{Before adjustments.}
    \label{RowBefore}
    \end{figure}

    \begin{figure}[h!]
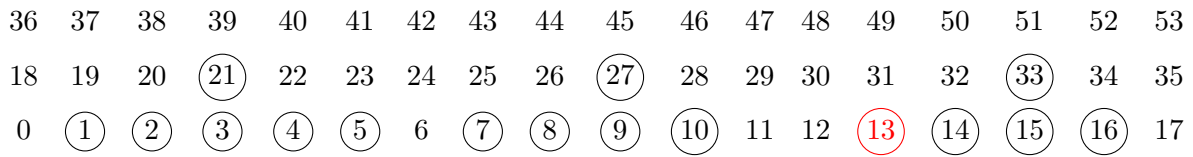

    \centering
    \[ \small
    \begin{array}{cccccccccccccccccc}
    36 & 37 & 38 & 39 & 40 & 41 & 42 & 43 & 44 & 45 & 46 & 47 & 48 & 49 & 50 & 51 & 52 & 53 \\
    18 & 19 & 20 & \encircle{21} & 22 & 23 & 24 & 25 & 26 & \encircle{27} & 28 & 29 & 30 & 31 & 32 & \encircle{33} & 34 & 35\\
    0 & \encircle{1} & \encircle{2} & \encircle{3} & \encircle{4} & \encircle{5} & 6 & \encircle{7} & \encircle{8} & \encircle{9} & \encircle{10} & 11 & 12 & \enred{13} & \encircle{14} & \encircle{15} & \encircle{16} & 17\\
    
    \end{array}
    \]
    \caption{After adjustments.}
    \label{RowAfter}
    \end{figure}

\end{example}

\section{Adjustments in different rows}

\par In this section, we aim to prove that if a generalized-$\beta$-set $S\subseteq\mathcal{L}_m(s)$ maximizes $f(S)$, then for nonzero $a_i$ and $a_{i+m}$,  we have $a_{i}=a_{i+m}+2$, 
thus the equality of Condition (3) in Definition \ref{def:4.2} holds.

\begin{lemma}\label{Normal}
Let $S$ be a generalized-$\beta$-set that maximizes $f(S)$. 
Then we have the following conclusions:
\begin{enumerate}

    \item[(1)] If $A_{t-1}-A_t\ge 2m$, then
\[
a_{i}=a_{i+m}+2
\]for $1\leq i\leq (t-1)m$.

\item[(2)] If $A_{t-1}-A_t\leq 2m-1$, by Proposition \ref{<2m}, assume that $a_{(t-1)m+1}=\cd=a_{(t-1)m+p}=1>a_{(t-1)m+p+1}$. Then
\[
a_{i}=a_{i+m}+2
\]for $1\leq i\leq (t-2)m+p$ and $a_{(t-2)m+p}=3>a_{(t-2)m+p+1}$.

\end{enumerate}

\end{lemma}

\begin{proof}
\par By Lemma \ref{row}, we know that $S\in\mathbb{E}_m(s) $.
\par \tb{Step 1.} First, we prove the following weaker version of Lemma \ref{Normal}:

\begin{enumerate}
    \item[(1)'] If $A_{t-1}-A_t\ge 2m$, then $2m \leq A_i-A_{i+1} \leq 2m+1$ for $1\leq i\leq t-1$. The equation $ A_i-A_{i+1} = 2m+1$ only holds by at most one $i$.
    \item[(2)'] If $A_{t-1}-A_t<2m$, then $2m \leq A_i-A_{i+1} \leq 2m+1$ for $1\leq i\leq t-2$. The equation $ A_i-A_{i+1} = 2m+1$ only holds by at most one $i$.
\end{enumerate}

\par First, we deal with the case $A_{t-1}-A_t\ge 2m$.  {We prove Conclusion (1)' by contradiction. Assume that Conclusion (1)' is not true}. There are two possible cases.

\begin{enumerate}

    \item[(1)] {There exists ${1\leq i\leq (t-1)m}$ such that ${A_{i}-A_{i+1} \ge 2m+2}$}. By Lemma \ref{row}, $0\leq a_{(i-1)m+1}-a_{im}\leq 1$ and $0\leq a_{im+1}-a_{(i+1)m}\leq 1$. Let $S'\in\mathbb{D}_m(s)$ satisfy $(n_k',a_k')=(n_k,a_k)$ for $k<(i-1)m+1$ and $k> (i+1)m$, and furthermore satisfy the following conditions:
    \begin{enumerate}
        \item If $a_{(i-1)m+1}=a_{im}$, then let $\left(n_{(i-1)m+l}',a_{(i-1)m+l}'\right)=\left(n_{(i-1)m+l}+1,a_{(i-1)m+l}\right)$ for $1\leq l\leq m-1$ and $\left(n_{im}',a_{im}'\right)=\left(n_{im},a_{im}-1\right)$;
        \item If $a_{(i-1)m+1}=a_{im}+1$, assume that $a_{im}=a_{im-1}=\cdots=a_{im-u+1}<a_{im-u}$. Let $\left(n_{im-l}',a_{im-l}'\right)=\left(n_{im-l},a_{im-l}\right)$ for $0\leq l\leq m-1$ and $l\neq u$  and $\left(n_{im-u}',a_{im-u}'\right) \\ =\left(n_{im-u}-1,a_{im-u}-1\right)$;
        \item If $a_{(i+1)m-1}=a_{(i+1)m}+1$, let $\left(n_{im+l}',a_{im+l}'\right)=\left(n_{im+l}-1,a_{im+l}\right)$ for $1\leq l\leq m-1$ and $\left(n_{(i+1)m}',a_{(i+1)m}'\right)=\left(n_{(i+1)m},a_{(i+1)m}+1\right)$;
        \item If $a_{(i+1)m-1}=a_{(i+1)m}$ and $a_{im+1}=a_{(i+1)m}+1$, then we can assume that $a_{(i+1)m}=a_{(i+1)m-1}=\cdots=a_{(i+1)m-v}<a_{(i+1)m-v-1}$ for positive integer $v$. Let \[\left(n_{(i+1)m-v}',a_{(i+1)m-v}'\right)= 
 \left(n_{(i+1)m-v}+1,a_{(i+1)m-v}+1\right)\]and $\left(n_{(i+1)m-l}',a_{(i+1)m-l}'\right)=\left(n_{(i+1)m-l},a_{(i+1)m-l}\right)$ for $0\leq l\leq m-1$ and $l\neq v$;
        \item If $a_{im+1}=a_{(i+1)m}$, let \[
\left(n_{im+1}',a_{im+1}'\right)=\left(n_{im+1}+1,a_{im+1}+1\right)
        \] and $\left(n_{im+l}',a_{im+l}'\right)$
        $=\left(n_{im+l},a_{im+l}\right)$ for $2\leq l\leq m$.
    \end{enumerate}Adjustments in the $i$-th row involve Cases (a) and (b). Adjustments in the $i+1$-th row involve Cases (c), (d) and (e). Obviously $S'\in\mathbb{E}_m(s) $. If (c) doesn't hold, then $f(S)<f(S')$.  
    {If (c) holds, there are two possibilities: (b) and (c) hold simultaneously or (a) and (c) hold simultaneously.} Recall that $D(i)$ denote $\sum_{x\in\mathcal{B}_i'}x-\sum_{x\in\mathcal{B}_i}x$. 
        
        \par  {Assume that (b) and (c) hold simultaneously}. Notice that $a_{(i+1)m} \leq s-2-i$, thus
        \begin{eqnarray*}
        && f(S')-f(S) \\
        & = & D\left((i+1)m\right)+D\left(im-u\right)+\sum_{j=1}^{m-1} D\left(im+j\right)  \\
        & \ge & ms - (m-1)\left(a_{(i+1)m}+1\right)\\
        & \ge & ms-(m-1)(s-1-i) \\
        & > & ms-ms=0.
        \end{eqnarray*}
        {Therefore, ${f(S')>f(S)}$. A contradiction}.
        
         \par {Assume that (a) and (c) hold simultaneously}. Similarly,
         
         \begin{eqnarray*}
         && f(S')-f(S) \\
         & = & D(im)+D((i+1)m)+\sum_{k=1}^{m-1}D((i-1)m+k)+\sum_{k=1}^{m-1}D(im+k) \\
         & \ge & (m-1)s+(m-1)a_{im}-(m-1)\left(a_{(i+1)m}+1\right) \\ 
         & \ge & (m-1)s+(m-1)\left(a_{(i+1)m}+2\right)-(m-1)\left(a_{(i+1)m}+1\right) \\
         & > & 0.
         \end{eqnarray*}
          {Therefore, ${f(S')>f(S)}$. A contradiction}.
    
    \item[(2)] {There exists ${i<j}$, such that ${A_i-A_{i+1} \ge 2m+1}$ and ${A_j-A_{j+1} \ge 2m+1}$}. Then we construct $S'\in\mathbb{D}_m(s) $ similar to Case (1). We first do the adjustments to the $i$-th row of $S$ similar to the adjustments to the $i$-th row in Case (1), and then do the adjustments to the $(j+1)$-th row of $S$ similar to the adjustments to the $i+1$-th row in Case (1). Then we obtain $S'$ after the adjustment. It is obvious that $S'\in\mathbb{E}_m(s)$. {Similar to the discussion in Case (1), ${f(S')>f(S)}$. A contradiction}.
    
\end{enumerate}

\par Next, we deal with the case $A_{t-1}-A_t<2m$.    In this case, by Proposition \ref{<2m}, there exists $1\leq p<m$, such that $a_{(t-1)m}=a_{(t-1)m+1}=\cdots=a_{(t-1)m+p}=1$ and $a_{(t-1)m+p+1}=\cdots=a_{tm}=0$. 
{We will prove Conclusion (2)' by contradiction. Assume that Conclusion (2)' isn't true}. Then either there exists $1\leq i\leq t-2$, such that $A_{i}-A_{i+1}\ge 2m+2$; or there exist $1\leq i<j\leq t-2$, such that $A_i-A_{i+1}=2m+1$ and $A_j-A_{j+1}=2m+1$. We will deduce a contradiction for each case.

Define
\[
\mathcal{L}_i=\bigcup_{j=1}^p\mathcal{B}_{(i-1)m+j}
\]
and
\[
\mathcal{R}_i=\bigcup_{j=p+1}^{m} \mathcal{B}_{(i-1)m+j}.
\]for $1\leq i\leq t$. Then their cardinal numbers are
\[
L_i=|\mathcal{L}_i|=\sum_{j=1}^{p} a_{(i-1)m+j},
\]
\[
R_i=|\mathcal{R}_i|=\sum_{j=p+1}^{m} a_{(i-1)m+j}.
\]

\par {Firstly, we deal with the case ${A_i-A_{i+1} \ge 2m+2}$}. We have the following three cases to consider.

\begin{enumerate}
    \item[(1)] $L_i-L_{i+1}\ge 2p+1$ and $R_i-R_{i+1}\ge 2(m-p)+1$. 
    
    
    \par Consider $S'\in\mathbb{D}_m(s)$ satisfying the following conditions:
    \begin{enumerate}
    \item $\left(n_k',a_k'\right)=(n_k,a_k)$ unless $(i-1)m+1\leq k\leq (i+1)m$;
        \item If $a_{(i-1)m+p+1}=a_{im}$, assume that $a_{(i-1)m+p+1}=a_{(i-1)m+p}=\cd=a_{(i-1)m+u}<a_{(i-1)m+u-1} $. Then $\left(n_k',a_k'\right)=(n_k+1,a_k)$ for $(i-1)m+u\leq k\leq im-1$ and $\left(n_{im}',a_{im}'\right)=(n_{im},a_{im}-1)$. Let $\left(n_k',a_k'\right)=(n_k,a_k)$ for $(i-1)m+1 \leq k\leq (i-1)m+u-1$.
        \item If $a_{(i-1)m+p+1}=a_{im}+1$, assume that $a_{im}=a_{im-1}=\cdots=a_{im-v-1}<a_{im-v}$. For $(i-1)m+1\leq k\leq im$, $\left(n_k',a_k'\right)=(n_k,a_k)$ if $k\neq im-v$ and $\left(n_{im-v}',a_{im-v}'\right)=\left(n_{im-v}-1,a_{im-v}-1\right)$;
        \item If $a_{im+1}=a_{(i+1)m}$, then $\left(n_{im+1}',a_{im+1}'\right)=(n_k+1,a_k+1)$ and $\left(n_k',a_k'\right)=(n_k,a_k)$ for $im+2\leq k\leq (i+1)m$;
        \item If $a_{im+1}=a_{im+p}>a_{(i+1)m}$, then $\left(n_{im+1}',a_{im+1}'\right)=(n_k,a_k+1)$ and $\left(n_k',a_k'\right)=(n_k,a_k)$ for $im+2\leq k\leq (i+1)m$;
        \item If $a_{im+1}>a_{im+p}=a_{(i+1)m}$, assume that $a_{im+1}=a_{im+2}=\cdots=a_{im+u-1}>a_{im+u}$. For $im+1\leq k\leq (i+1)m$, $\left(n_k',a_k'\right)=(n_k,a_k)$ if $k\neq im+u$ and $\left(n_{im+u}',a_{im+u}'\right)=\left(n_{im+u}+1,a_{im+u}+1\right)$;
        \item If $a_{im+1}>a_{im+p}>a_{(i+1)m}$, assume that $a_{im+1}=a_{im+2}=\cdots=a_{im+u-1}>a_{im+u}$. For $im+1\leq k\leq (i+1)m$, $\left(n_k',a_k'\right)=(n_k,a_k)$ if $k\neq im+u$ and $\left(n_{im+u}',a_{im+u}'\right)=\left(n_{im+u},a_{im+u}+1\right)$.
    \end{enumerate}
    Cases (b) and (c) are adjustments in $\mathcal{R}_i$. Cases (d),(e),(f) and (g) are adjustments in $\mathcal{L}_{i+1}$. Then $L_i'-L_{i+1}'\ge 2p$ and $R_i'-R_{i+1}'\ge 2(m-p)$. By Lemma \ref{One Line}, $a_{im+k}'\leq a_{(i-1)m+k}'-2$ for $1\leq k\leq m$. Condition (3) in Definition \ref{def:4.2} won't be violated after the adjustments. We can verify that $S'\in\mathbb{E}_m(s) $. Similar to the discussion for the case $A_{t-1}-A_t\ge 2m$, we obtain $|S|=|S'|$ and $f(S)<f(S')$. {Thus we reach a contradiction}.
    
    \item[(2)] {${L_i-L_{i+1}=2p}$ and ${R_i-R_{i+1}\ge 2(m-p)+2}$}. 
    
    
    \par Consider $S'\in\mathbb{D}_m(s)$ satisfying the following conditions:
    \begin{enumerate}
        \item $(n_k',a_k')=(n_k,a_k)$ unless $(i-1)m+1\leq k\leq (i+1)m$;
        \item If $a_{(i-1)m+p+1}=a_{im}$, assume that $a_{(i-1)m+u-1}>a_{(i-1)m+u}=\cd=a_{im-1}=a_{im} $. then $(n_k',a_k')=(n_k,a_k)$ for $(i-1)m+1\leq k\leq (i-1)m+u-1$, $(n_k',a_k')=(n_k+1,a_k)$ for $(i-1)m+u\leq k\leq im-1$ and $\left(n_{im}',a_{im}'\right)=\left(n_{im},a_{im}-1\right)$;
        \item If $a_{(i-1)m+p+1}=a_{im}+1$, assume that $a_{im}=a_{im-1}=\cdots=a_{im-v-1}<a_{im-v}$. For $(i-1)m+1\leq k\leq im$, $(n_k',a_k')=(n_k,a_k)$ if $k\neq im-v$ and $\left(n_{im-v}',a_{im-v}'\right)=\left(n_{im-v}-1,a_{im-v}-1\right)$;
        \item If $a_{im+p+1}=a_{(i+1)m}$, for $im+1\leq k\leq (i+1)m$, let $\left(n_k',a_k'\right)=(n_k,a_k)$ if $k\neq im+p+1$ and $\left(n_{im+p+1}',a_{im+p+1}'\right)=\left(n_{im+p+1}+1,a_{im+p+1}+1\right)$;
        \item If $a_{(i+1)m-1}=a_{(i+1)m}$ and $a_{im+p+1}=a_{(i+1)m}+1$, assume that $a_{(i+1)m}=a_{(i+1)m-1}=\cdots=a_{(i+1)m-v}<a_{(i+1)m-v-1}$. For $im+1\leq k\leq (i+1)m$, $(n_k',a_k')=(n_k,a_k)$ if $k\neq im-v$ and
        \[
        \left(n_{(i+1)m-v}',a_{(i+1)m-v}'\right)=\left(n_{(i+1)m-v}+1,a_{(i+1)m-v}+1\right);
        \]
        \item If $a_{im+p+1}=a_{(i+1)m-1}=a_{(i+1)m}+1$, assume that $a_{(i+1)m-1}=a_{(i+1)m-2}=\cdots=a_{(i+1)m-w}<a_{(i+1)m-w-1}$. Let $\left(n_{(i+1)m}',a_{(i+1)m}'\right)=\left(n_{(i+1)m},a_{(i+1)m}+1\right)$ and $\left(n_{k}',a_{k}'\right)=\left(n_{k}-1,a_{k}\right)$ for $(i+1)m-w \leq k \leq (i+1)m-1$.
    \end{enumerate}
    Case (b) and (c) are adjustments in $\mathcal{R}_i$. Case (d), (e) and (f) are adjustments in $\mathcal{R}_{i+1}$. We claim that Case (d) is valid. Otherwise, $S'$ violates Condition (2) in Definition \ref{def:4.2}, then
    \[
    a_{im+p}=a_{im+p+1}=\cdots=a_{(i+1)m}.
    \]Notice that $L_i-L_{i+1}=2p$ and $R_i-R_{i+1}\ge 2(m-p)+2$, we have $a_{(i-1)m+p}<a_{(i-1)m+p+1}$. A contradiction. Therefore, $S'$ satisfies Condition (2) in Definition \ref{def:4.2}. We can easily prove that $S'\in\mathbb{E}_m(s) $. When (f) doesn't happen, obviously $f(S)<f(S')$. When (f) happens, similar to the discussions in Case (1) when $A_{t-1}-A_t \ge 2m$, we can prove that $f(S)<f(S')$. {A contradiction}.
    
    \item[(3)] {${L_i-L_{i+1} \ge 2p+2}$ and ${R_i-R_{i+1}=2(m-p)}$}. It is similar to Case (2).

\end{enumerate}

\par Therefore, we obtain $2m\leq A_i-A_{i+1}\leq 2m+1$ for $1\leq i\leq t-2$. {Now we deal with the case that there exists ${1\leq i<j\leq t-2}$, such that ${A_i-A_{i+1}=2m+1}$ and ${A_j-A_{j+1}=2m+1}$}. There are four cases to be considered.

\begin{enumerate}
    \item[(1)] {${R_i-R_{i+1}=2(m-p)+1}$ and ${L_j-L_{j+1}=2p+1}$}. This case is similar to Case (1) when $A_i-A_{i+1}\ge 2m+2$. {We can still reach a contradiction}. 
    
    
    \item[(2)] $L_i-L_{i+1}=L_j-L_{j+1}=2p$ and $R_i-R_{i+1}=R_j-R_{j+1}=2(m-p)+1$. This case is similar to Case (2) when $A_i-A_{i+1}\ge 2m+2$. {We can still reach a contradiction}. 
    
    
    \item[(3)] $L_i-L_{i+1}=L_j-L_{j+1}=2p+1$ and $R_i-R_{i+1}=R_j- R_{j+1}=2(m-p)$. This case is similar to Case (3) when $A_i-A_{i+1}\ge 2m+2$. {We can still reach a contradiction}. 
    
    
    \item[(4)] {$L_i-L_{i+1}=2p+1$, $R_i-R_{i+1}=2(m-p)$, $L_j-L_{j+1}=2p$, $R_j-R_{j+1}=2(m-p)+1$}. 

    \par Assume that $a_{(i-1)m+1}=a_{(i-1)m+p+1}$. Since $L_i-L_{i+1}=2p+1$ and $R_i-R_{i+1}=2(m-p)$, $a_{im+p}=a_{(i-1)m+p}-3=a_{(i-1)m+p+1}-3=a_{im+p+1}-1<a_{im+p+1} $. A contradiction. Thus $a_{(i-1)m+1}>a_{(i-1)m+p+1}$. Similarly, $a_{jm+p}>a_{(j+1)m}$.

    \par Consider $S'\in\mathbb{D}_m(s)$ satisfying the following conditions:
    \begin{enumerate}
        \item $(n_k',a_k')=(n_k,a_k)$ unless $(i-1)m+1\leq k\leq im$ or $jm+1\leq k\leq (j+1)m$ ;
        \item If $a_{(i-1)m+1}=a_{(i-1)m+p}=a_{im}+1$, then $(n_k',a_k')=(n_k,a_k)$ for $(i-1)m+1\leq k\leq (i-1)m+p-1$ or $(i-1)m+p+1\leq k\leq im$ and $\left(n_{(i-1)m+p}',a_{(i-1)m+p}'\right)=\left(n_{(i-1)m+p}-1,a_{(i-1)m+p}-1\right)$;
        \item If $a_{(i-1)m+1}=a_{(i-1)m+p}>a_{im}+1$, then $(n_k',a_k')=(n_k,a_k)$ for $(i-1)m+1\leq k\leq (i-1)m+p-1$ or $(i-1)m+p+1\leq k\leq im$ and $\left(n_{(i-1)m+p}',a_{(i-1)m+p}'\right)=\left(n_{(i-1)m+p},a_{(i-1)m+p}-1\right)$;
        \item If $a_{(i-1)m+1}=a_{(i-1)m+p}+1=a_{im}+1$, assume that $a_{(i-1)m+1}=a_{(i-1)m+2}=\cdots=a_{(i-1)m+u}>a_{(i-1)m+u+1}$. For $(i-1)m+1\leq k\leq im$, $(n_k',a_k')=(n_k,a_k)$ if $k\neq (i-1)m+u$ and $\left(n_{(i-1)m+u}',a_{(i-1)m+u}'\right)=\left(n_{(i-1)m+u}-1,a_{(i-1)m+u}-1\right)$;
        \item If $a_{(i-1)m+1}=a_{(i-1)m+p}+1>a_{im}+1$, assume that $a_{(i-1)m+1}=a_{(i-1)m+2}=\cdots=a_{(i-1)m+u}>a_{(i-1)m+u+1}$. For $(i-1)m+1\leq k\leq im$, $(n_k',a_k')=(n_k,a_k)$ if $k\neq (i-1)m+u$ and $\left(n_{(i-1)m+u}',a_{(i-1)m+u}'\right)=\left(n_{(i-1)m+u},a_{(i-1)m+u}-1\right)$;
        \item If $a_{jm+p+1}=a_{(j+1)m}$, for $jm+1\leq k\leq (j+1)m$, let $\left(n_k',a_k'\right)=(n_k,a_k)$ if $k\neq jm+p+1$ and $\left(n_{jm+p+1}',a_{jm+p+1}'\right)=\left(n_{jm+p+1}+1,a_{jm+p+1}+1\right)$;
        \item If $a_{(j+1)m-1}=a_{(j+1)m}$ and $a_{jm+p+1}=a_{(j+1)m}+1$, assume that $a_{(j+1)m}=a_{(j+1)m-1}=\cdots=a_{(j+1)m-v}<a_{(j+1)m-v-1}$. For $jm+1\leq k\leq (j+1)m$, $(n_k',a_k')=(n_k,a_k)$ if $k\neq jm-v$ and
        \[
        \left(n_{(j+1)m-v}',a_{(j+1)m-v}'\right)=\left(n_{(j+1)m-v}+1,a_{(j+1)m-v}+1\right);
        \]
        \item If $a_{jm+p+1}=a_{(j+1)m-1}=a_{(j+1)m}+1$, assume that $a_{(j+1)m-1}=a_{(j+1)m-2}=\cdots=a_{(j+1)m-w}<a_{(j+1)m-w-1}$. Let $\left(n_{(j+1)m}',a_{(j+1)m}'\right)=\left(n_{(j+1)m},a_{(j+1)m}+1\right)$ and $\left(n_{k}',a_{k}'\right)=\left(n_{k}-1,a_{k}\right)$ for $(j+1)m-w \leq k \leq (j+1)m-1$.

    \end{enumerate}

    Cases (b), (c), (d) and (e) are adjustments in $\mathcal{L}_i$. Cases (f), (g), (h) are adjustments in $\mathcal{R}_{j+1}$. By the discussion before, Cases (b) and (f) are valid. Similarly we can prove that $S'\in\mathbb{E}_m(s) $ and $f(S)<f(S')$. { A contradiction}.
    
    
\end{enumerate}

\par \tb{Step 2.} 
Now we provide the proofs of Conclusions (1)(2). First, we prove Conclusion (1) by contradiction. Assume that
$
{a_{i}=a_{i+m}+2}
$ 
for ${1\leq i\leq (t-1)m}$ doesn't hold. Then by Conclusion (1)', there exists a unique ${1\leq l\leq t-1}$, such that ${A_l-A_{l+1}=2m+1}$. Thus there exists a unique $u$, such that $a_u-a_{u+m}=3$ and $a_i-a_{i+m}=2$ for $i\neq u$. Assume that $q\leq m$ is the largest positive integer such that $a_1=\cdots=a_q$. By Condition (2) in Definition \ref{def:4.2} and Conclusion (2) in Lemma \ref{row}, we obtain
\[
    u\equiv q \ (\text{mod $m$}).
\]
Therefore, $u=(l-1)m+q$. By Lemma \ref{row}, $a_{(i-1)m+1}-a_{im}\leq 1$ for $1\leq i\leq t$, so we have two cases: $a_1=a_m+1$ or $a_1=a_m$.

\par Firstly, consider the case ${a_1=a_m+1}$. For $0\leq w\leq t$, let $S_w\in\mathbb{D}_m(s)$ satisfy the following conditions. Here we use $(n_i^w,a_i^w)$ to denote $\left(n_i\left(S_w\right),a_i\left(S_w\right)\right)$.
\begin{enumerate}
    \item[(1)] $\left(n_i^w,a_i^w\right)=(n_i,a_i)$ 
    for $i\neq q,2q,\cdots,(t-1)m+q$; 
    \item[(2)] If $0\leq k\leq l-1$, then $(n_i^w,a_i^w)=(n_i-1,a_i-1)$ for $i=km+q,(k+1)m+q,\cdots,(l-1)m+q$;
    \item[(3)] If $l\leq k\leq t$, then $(n_i^w,a_i^w)=(n_i+1,a_i+1)$ for $i=lm+q,(l+1)m+q,\cdots,(k-1)m+q$.
\end{enumerate}Then $S=S_l$. It's easy to check that $S_0,\cd,S_t\in\mathbb{E}_m(s) $ and $|S_i|-|S_{i-1} |=1$ for $1\leq i\leq t$. Since\begin{eqnarray*}
    f(S_{n+1})-f(S_n) &=& \sum_{x\in S_{n+1}} x -\frac{| S_{n+1}|\left(| S_{n+1}|-1\right)}{2}-\sum_{x\in S_{n}} x + \frac{| S_{n}|(| S_{n}|-1)}{2} \\
    &=& \sum_{x\in S_{n+1}} x-\sum_{x\in S_n}x - \frac{| S_{n}|(| S_{n}|+1)}{2}+\frac{| S_{n}|(| S_{n}|-1)}{2} \\
    &=& \sum_{x\in S_{n+1}} x-\sum_{x\in S_n}x - | S_n|,
\end{eqnarray*}
we have
\begin{eqnarray*}
        && f(S_{n+2})-2f(S_{n+1})+f(S_n) \\ &=&\sum_{x\in S_{n+2}-S_{n+1}} x -\sum_{x\in S_{n+1}-S_n} x-(| S_{n+1}| - | S_n|)\\
        &=& ms-1-1\\
        &=& ms-2>0.
        \end{eqnarray*}
        Thus $f(S_{n+2})-f(S_{n+1})>f(S_{n+1})-f(S_{n})$, and $f(S)=f(S_l)<\max\{f(S_{l-1}),f(S_{l+1})\}$. A {contradiction}.
        
\par Secondly, consider the case ${a_1=a_m}$. For $0\leq w\leq t$, let $S_w\in\mathbb{D}_m(s)$ satisfy the following conditions. 
Here we still use $(n_i^w,a_i^w)$ to denote $\left(n_i\left(S_w\right),a_i\left(S_w\right)\right)$.
\begin{enumerate}
    \item[(1)] $(n_i^w,a_i^w)=(n_i,a_i)$ except when $i=q,2q,\cdots,(t-1)m+q$;
    \item[(2)] Assume that $0\leq k\leq l-1$. For $km+1\leq i\leq lm$, $(n_i^w,a_i^w)=(n_i,a_i-1)$ if $i\equiv q\mod m$ and $(n_i^w,a_i^w)=(n_i+1,a_i)$ if $i\not\equiv q\mod m$. Otherwise, $(n_i^w,a_i^w)=(n_i,a_i)$.
    \item[(3)] Assume that $l\leq k\leq t$. For $lm+1\leq i\leq km$, $(n_i^w,a_i^w)=(n_i,a_i+1)$ if $i\equiv q\mod m$ and $(n_i^w,a_i^w)=(n_i-1,a_i)$ if $i\not\equiv q\mod m$. Otherwise, $(n_i^w,a_i^w)=(n_i,a_i)$. 
\end{enumerate}

Then $S=S_l$. It's easy to check that $S_0,\cd,S_t\in\mathbb{E}_m(s) $ and $|S_i|-|S_{i-1} |=1$ for $1\leq i\leq t$. Similarly,
    \begin{eqnarray*}
        && f(S_{n+2})-2f(S_{n+1})+f(S_n) \\ &=& \sum_{x\in S_{n+2}} x - \sum_{x\in S_{n+1}} x - \left(\sum_{x\in S_{n+1}} x - \sum_{x\in S_{n}} x\right)-\left|S_{n+1}|+|S_n\right| \\
        &=& ms+1+2(m-1)-1=m(s+2)-2>0.
        \end{eqnarray*}
        This also contradicts the assumption of $S$.

\par Now we turn to the case where $A_{t-1}-A_t<2m$. 
{We prove Conclusion (2) by contradiction. First, assume that the claim
$
{a_{i}=a_{i+m}+2}
$ 
for ${1\leq i\leq (t-2)m}$ is not true.} Similarly, there exist unique ${1\leq l\leq t-1}$ and ${1\leq k\leq m}$, such that ${a_{(l-1)m+k}-a_{lm+k}=3}$ and $a_i-a_{i+m}=2$ for $i\neq (l-1)m+k$. We can make adjustments similar to the former case to the first $p$ columns of $S$ and the last $m-p$ columns of $S$ separately. Still, we can reach a {contradiction}. Thus $a_i-a_{i+m}=2$ for $1\leq i\leq (t-2)m$. 
\par Next, we prove that $a_{(t-2)m+1}=\cd=a_{(t-2)m+p}=3>a_{(t-2)m+p+1}$ by contradiction. We know that $a_{(t-2)m+p}\ge a_{(t-1)m+p}+2=3 $. {Assume that $a_{(t-2)m+p+1}\ge 3$, then $p<m-1$ since $a_{(t-1)m}=1$.} Let $S'=S\cup\{((t-1)m+p+1)s-t\}$, then $S'$ is still a generalized-$\beta$-set and $f(S)<f(S')$. A contradiction. Thus $a_{(t-2)m+p+1}\leq 2$. {Assume that $a_{(t-2)m+1}\ge 4$. Set $a_{(t-2)m+1}=a_{(t-2)m+2}=\cdots=a_{(t-2)m+u}>a_{(t-2)m+u+1}$ for some $1\leq u\leq p$ and $a_{(t-1)m}=a_{(t-1)m-1}=\cdots=a_{(t-1)m-v}<a_{(t-1)m-v-1}$ for some $0 \leq v \leq m-p-1$.} {Since $A_{t-1}-A_t < 2m$, we have $p < m-1$ and $v \ge 1$.} Let $S'\in\mathbb{D}_m(s)$ satisfy the following conditions:
\begin{enumerate}
    \item[(1)] $(n_i',a_i')=(n_i,a_i)$ if $i \equiv 1, 2, \ldots, u-1, u+1, \ldots, p\ (\text{mod}\ m)$;
    \item[(2)] $\left(n_{jm+u}',a_{jm+u}'\right)=\left(n_{jm+u},a_{jm+u}-1\right)$ for $0 \leq j \leq t - 2$;
    \item [(3)] Let 
    \[
    \left(n_{jm-v}',a_{jm-v}'\right)=\left(n_{jm-v}+1,a_{jm-v}+1\right)
    \]
    for $1 \leq j \leq t-1$ and $\left(n_{(j-1)m+i}',a_{(j-1)m+i}'\right) =\left(n_{(j-1)m+i},a_{(j-1)m+i}\right)$ if $1 \leq j \leq t-1$, $p+1 \leq i \leq m$ and $i \neq m-v$.
\end{enumerate}
Since $a_i-a_{i+m}=2$ for $1\leq i\leq (t-2)m$, it's easy to verify that $S'\in\mathbb{E}_m(s) $. If (4) doesn't happen, obviously $f(S)<f(S')$, a contradiction. Otherwise, $v=0$ and $a_{im}=a_{im-1}-1$ for $1 \leq i \leq t-1$. Recall that $D(i)$ denote $\sum_{x\in\mathcal{B}_i'}x-\sum_{x\in\mathcal{B}_i}x$. Then
\begin{align*}
    f(S') - f(S) &= \sum_{x \in S'} x - \sum_{x \in S} x \\
    &= \sum_{i=0}^{t-2} \left( D(im+u) + D ( (i+1)m )  \right) + \sum_{i=0}^{t-2} \sum_{j=p+1}^{m-1} D(im+j) \\
    & \ge (t-1)(m-u)s - (t-1) (m-p-1) s >0.
\end{align*}So $f(S')>f(S)$. A contradiction. Thus for $(t-2)m+1\leq i\leq (t-2)m+p$, we have $a_i-a_{i+m}=2$.
\end{proof}

\begin{definition}
    Let $\mathbb{F}_m(s)$ be the set of all $S\in\mathbb{E}_m(s)$ such that $S$ satisfies Conclusions (1) and (2) in Lemma \ref{Normal}.
\end{definition}

By Lemma \ref{Normal}, we obtain the following corollary.

\begin{corollary} \label{cor: 7.3}
We have $\mathbb{F}_m(s)\subseteq\mathbb{E}_m(s)\subseteq\mathbb{C}_m(s)\subseteq\mathbb{B}_m(s)$ and
\begin{equation*}
    {\arg}\max_{S'\in\mathbb{B}_m(s)}f(S')={\arg}\max_{S''\in\mathbb{C}_m(s)}f(S'')={\arg}\max_{S'''\in\mathbb{E}_m(s)}f(S''')={\arg}\max_{S''''\in\mathbb{F}_m(s)}f(S'''').
\end{equation*}
\end{corollary}

Next, we give an example illustrating the adjustments in the proof of Lemma \ref{Normal}.

\begin{example}
Let $s=6$ and $m=3$. Figure \ref{Before} shows the generalized-$\beta$-set $S$ before the adjustments. Figure \ref{Step1} shows the generalized-$\beta$-set $S'$ after the first adjustment. Figure \ref{Step21} shows the same generalized-$\beta$-set $S'$ with Figure \ref{Step1}. Figure \ref{Step2} shows the generalized-$\beta$-set $S''$ after the second adjustment. The differences in the colors of some circles between Figure \ref{Step1} and \ref{Step21} are to highlight the changes of different adjustments.  Here $t=2, A_1=13, A_2=3$. Thus $A_{t-1}-A_t\ge 2m$. 

Notice that $S'=S\cup\{22\}\setminus\{5\}$ is a generalized-$\beta$-set. We color $S\setminus S'$ blue and $S'\setminus S$ red. The first adjustment is in rows $1$ and $2$ and involves (1)(b) and (1)(e) in the case $A_{t-1}-A_t\ge 2m$ of Step 1 of the proof in Pages 15 and 16. We can find that $(n_1,a_1)=(5,5),(n_4,a_4)=(1,3),(n_1',a_1')=(4,4),(n_4',a_4')=(2,4)$. Indeed, we have $(n_1',a_1')=(n_1-1,a_1-1)$ and $(n_4',a_4')=(n_4+1,a_4+1)$. 
Obviously $f(S')-f(S)=22-5>0.$ Here $A_1'=12,A_2'=4$. So $A'_{t-1}-A'_t\ge 2m$. 

Notice that $S''=S'\cup\{5,11,28\}\setminus\{1,7,13\}$ is a generalized-$\beta$-set. We color $S'\setminus S''$ blue and $S''\setminus S'$ red. The second adjustment is in rows $1$ and $2$ and involves (1)(a) and (1)(d) in the case $A'_{t-1}-A'_t\ge 2m$ of Step 1 of the proof in Lemma \ref{Normal}. We can find that $(n_1'',a_1'')=(n'_1+1,a'_1), (n_2'',a_2'')=(n'_2+1,a'_2), (n_3'',a_3'')=(n'_3,a'_3-1)$ and $(n_5'',a_5'')=(n'_5+1,a'_5+1)$. Obviously $f(S'')-f(S')=(5-1)+(11-7)+(28-13)>0$.

    \begin{figure}[h!]
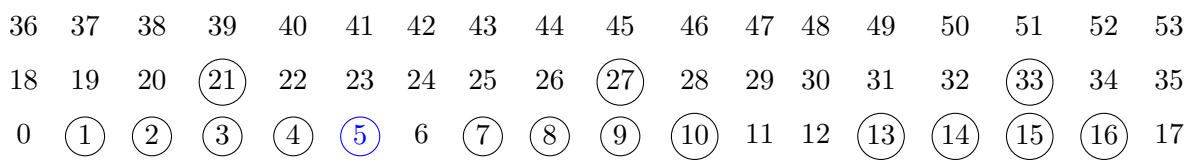

    \centering
    \[\small
    \begin{array}{cccccccccccccccccc}
    36 & 37 & 38 & 39 & 40 & 41 & 42 & 43 & 44 & 45 & 46 & 47 & 48 & 49 & 50 & 51 & 52 & 53 \\
    18 & 19 & 20 & \encircle{21} & 22 & 23 & 24 & 25 & 26 & \encircle{27} & 28 & 29 & 30 & 31 & 32 & \encircle{33} & 34 & 35\\
    0 & \encircle{1} & \encircle{2} & \encircle{3} & \encircle{4} & \enblue{5} & 6 & \encircle{7} & \encircle{8} & \encircle{9} & \encircle{10} & 11 & 12 & \encircle{13} & \encircle{14} & \encircle{15} & \encircle{16} & 17\\
    
    \end{array}
    \]
    \caption{Before adjustments.}
    \label{Before}
    \end{figure}

    \begin{figure}[h!]
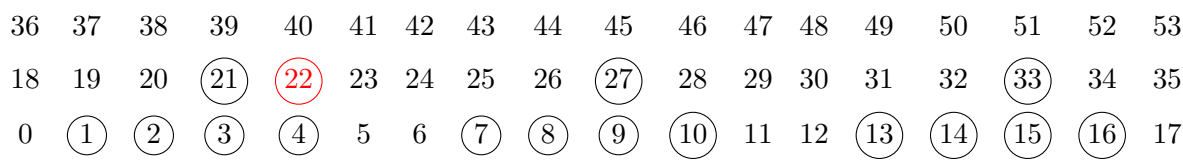

    \centering
    \[ \small
    \begin{array}{cccccccccccccccccc}
    36 & 37 & 38 & 39 & 40 & 41 & 42 & 43 & 44 & 45 & 46 & 47 & 48 & 49 & 50 & 51 & 52 & 53 \\
    18 & 19 & 20 & \encircle{21} & \enred{22} & 23 & 24 & 25 & 26 & \encircle{27} & 28 & 29 & 30 & 31 & 32 & \encircle{33} & 34 & 35\\
    0 & \encircle{1} & \encircle{2} & \encircle{3} & \encircle{4} & 5 & 6 & \encircle{7} & \encircle{8} & \encircle{9} & \encircle{10} & 11 & 12 & \encircle{13} & \encircle{14} & \encircle{15} & \encircle{16} & 17\\
    
    \end{array}
    \]
    \caption{After the first adjustment.}
    \label{Step1}
    \end{figure}

    \begin{figure}[h!]
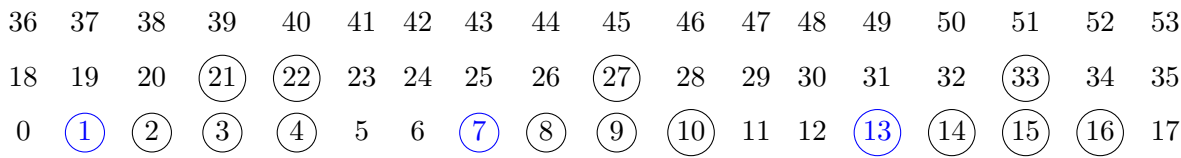

    \centering
    \[ \small
    \begin{array}{cccccccccccccccccc}
    36 & 37 & 38 & 39 & 40 & 41 & 42 & 43 & 44 & 45 & 46 & 47 & 48 & 49 & 50 & 51 & 52 & 53 \\
    18 & 19 & 20 & \encircle{21} & \encircle{22} & 23 & 24 & 25 & 26 & \encircle{27} & 28 & 29 & 30 & 31 & 32 & \encircle{33} & 34 & 35\\
    0 & \enblue{1} & \encircle{2} & \encircle{3} & \encircle{4} & 5 & 6 & \enblue{7} & \encircle{8} & \encircle{9} & \encircle{10} & 11 & 12 & \enblue{13} & \encircle{14} & \encircle{15} & \encircle{16} & 17\\
    
    \end{array}
    \]
    \caption{Before the second adjustment.}
    \label{Step21}
    \end{figure}

    \begin{figure}[h!]
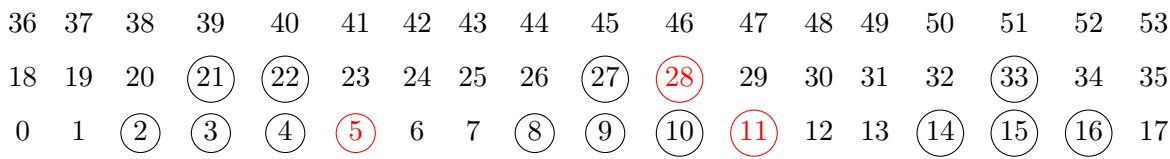

    \centering
    \[ \small
    \begin{array}{cccccccccccccccccc}
    36 & 37 & 38 & 39 & 40 & 41 & 42 & 43 & 44 & 45 & 46 & 47 & 48 & 49 & 50 & 51 & 52 & 53 \\
    18 & 19 & 20 & \encircle{21} & \encircle{22} & 23 & 24 & 25 & 26 & \encircle{27} & \enred{28} & 29 & 30 & 31 & 32 & \encircle{33} & 34 & 35\\
    0 & 1 & \encircle{2} & \encircle{3} & \encircle{4} & \enred{5} & 6 & 7 & \encircle{8} & \encircle{9} & \encircle{10} & \enred{11} & 12 & 13 & \encircle{14} & \encircle{15} & \encircle{16} & 17\\
    
    \end{array}
    \]
    \caption{After the second adjustment.}
    \label{Step2}
    \end{figure}

\end{example}

\section{Adjustments in columns}

\par In this section, first we prove two useful results Lemma \ref{Max1} and Lemma \ref{Max2}. These two lemmas deal with the cases $A_{t-1}-A_t\ge 2m$ and $A_{t-1}-A_t>2m$ respectively.


\begin{lemma}\label{Max1}
Let $S$ be a generalized-$\beta$-set that maximizes $f(S)$. Assume that $A_{t-1}-A_t\ge 2m$, then
\begin{enumerate}
\item[(1)] $a_1=\cdots=a_m$ or $a_1=\cdots=a_{m-1}=a_{m}+1$;
\item[(2)] $a_{(t-1)m+1}\in\{1,2\}$.
\end{enumerate}
\end{lemma}

\begin{proof}
\par By Lemma \ref{Normal}, we obtain that $S\in\mathbb{F}_m(s)$ and $a_i-a_{i+m}=2$ for $1\leq i\leq (t-1)m$. {We will prove Conclusion (1) by contradiction. Assume that Conclusion (1) is not true}. Then there exists $1\leq k\leq m-2$ and $a\ge 0$, such that $a_1=\cdots=a_k=a+1$ and $a_{k+1}=\cdots=a_m=a$.

    


\par For $0\leq n\leq m-1$, let $S_n\in\mathbb{C}_m(s)$ denote the $t$-row generalized-$\beta$-set such that $a_1=\cdots=a_n=a+1$, $a_{n+1}=\cdots=a_m=a$ and $a_i-a_{i+m}=2$ for all $i$. Then $S=S_k$ and we have

\begin{eqnarray*}
f(S_{n+1})-f(S_n) &=& \sum_{x\in S_{n+1}} x -\frac{|S_{n+1}|(|S_{n+1}|-1)}{2}-\sum_{x\in S_{n}} x + \frac{|S_{n}|(|S_{n}|-1)}{2} \\
&=& \sum_{x\in S_{n+1}} x - \sum_{x\in S_{n}} x-\sum_{i=|S_n|}^{|S_{n+1}|-1} i.
\end{eqnarray*}Therefore,
        
\begin{eqnarray*}
&& f(S_{n+2})-2f(S_{n+1})+f(S_n) \\ &=& \left(\sum_{x\in S_{n+2}} x - \sum_{x\in S_{n+1}} x\right) - \left(\sum_{x\in S_{n+1}} x - \sum_{x\in S_{n}} x\right)-\left(\sum_{i=|S_{n+1}|}^{|S_{n+2}|-1} i-\sum_{i=|S_n|}^{|S_{n+1}|-1} i\right) \\
&=& ts-t^2=t(s-t)>0.
\end{eqnarray*}
Thus we obtain $f(S)=f(S_k)<\max\{f(S_{k-1}),f(S_{k+1})\}$, a {contradiction}.

\par Next, {we prove Conclusion (2) by contradiction. Assume that Conclusion (2) is not true}. Then  $a_{(t-1)m+1}\ge 3$, thus $a_{tm}\ge a_{(t-1)m+1}-1\ge 2$. Let $b_i=|\mathcal {L}_m(s)\cap [(i-1)s, is-1]|$ for $i\ge 1$. Then $b_{tm+1}=b_{(t-1)m+1}-2\ge a_{(t-1)m+1}-2\ge 1$. Let $R=S\cup\{(tm+1)s-t-1\}$ and $T=S\cup\{(tm+1)s-t-2\}$. Then $R$ or $T$ is a generalized-$\beta$-set while $f(S)<f(R)$ and $f(S)<f(T)$. A {contradiction}.
\end{proof}

\begin{example}
We give an example illustrating the proof of Lemma \ref{Max1}. Let $s=5$ and $m=3$, consider $S_0$ in Figure \ref{fig:s0}, $S_1$ in Figure \ref{fig:s1} and $S_2$ in Figure \ref{fig:s2}. We color $S_0\setminus S_1$ blue and $S_2\setminus S_1$ red. Indeed $(a_1(S_0),a_2(S_0),a_3(S_0))=(3,3,3), (a_1(S_1),a_2(S_1),a_3(S_1))=(4,3,3), (a_1(S_2),a_2(S_2),a_3(S_2))=(4,4,3)$. It's easy to verify that $S_0, S_1, S_2 \in \mathbb{F}_m(s)$ with $f(S_0)=63$, $f(S_1)=60$ and $f(S_2)=63$. Therefore, $f(S_1) < \max\{f(S_0), f(S_2)\}$.
\begin{figure}[h!]
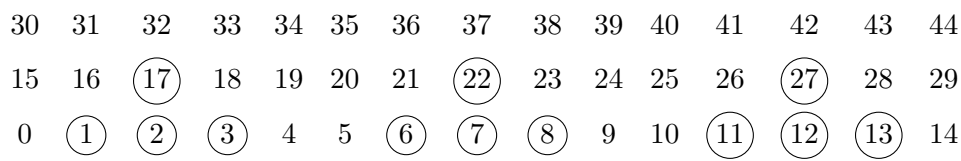

    \centering
    \[ \small
\begin{array}{ccccccccccccccc}
30 & 31 & 32 & 33 & 34 & 35 & 36 & 37 & 38 & 39 & 40 & 41 & 42 & 43 & 44\\
15 & 16 & \encircle{17} & 18 & 19 & 20 & 21 & \encircle{22} & 23 & 24 & 25 & 26 & \encircle{27} & 28 & 29\\
0 & \encircle{1} & \encircle{2} & \encircle{3} & 4 & 5 & \encircle{6} & \encircle{7} & \encircle{8} & 9 & 10 & \encircle{11} & \encircle{12} & \encircle{13} & 14\\
\end{array} 
\]
    \caption{$S_0$ with $(a_1(S_0),a_2(S_0),a_3(S_0))=(3,3,3)$.}
    \label{fig:s0}
\end{figure}

\begin{figure}[h!]
    \centering
    \[ \small
\begin{array}{ccccccccccccccc}
30 & 31 & 32 & 33 & 34 & 35 & 36 & 37 & 38 & 39 & 40 & 41 & 42 & 43 & 44\\
15 & 16 & \encircle{17} & \enblue{18} & 19 & 20 & 21 & \encircle{22} & 23 & 24 & 25 & 26 & \encircle{27} & 28 & 29\\
0 & \encircle{1} & \encircle{2} & \encircle{3} & \enblue{4} & 5 & \encircle{6} & \encircle{7} & \encircle{8} & 9 & 10 & \encircle{11} & \encircle{12} & \encircle{13} & 14\\
\end{array} 
\]
    \caption{$S_1$ with $(a_1(S_1),a_2(S_1),a_3(S_1))=(4,3,3)$.}
    \label{fig:s1}
\end{figure}

\begin{figure}[h!]
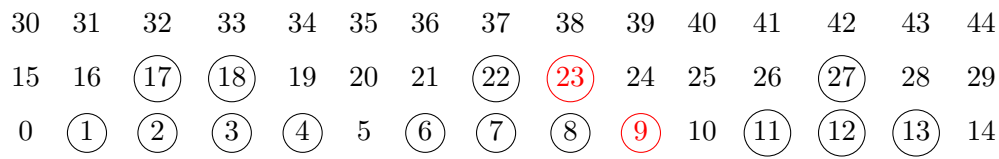

    \centering
    \[  \small
\begin{array}{ccccccccccccccc}
30 & 31 & 32 & 33 & 34 & 35 & 36 & 37 & 38 & 39 & 40 & 41 & 42 & 43 & 44\\
15 & 16 & \encircle{17} & \encircle{18} & 19 & 20 & 21 & \encircle{22} & \enred{23} & 24 & 25 & 26 & \encircle{27} & 28 & 29\\
0 & \encircle{1} & \encircle{2} & \encircle{3} & \encircle{4} & 5 & \encircle{6} & \encircle{7} & \encircle{8} & \enred{9} & 10 & \encircle{11} & \encircle{12} & \encircle{13} & 14\\
\end{array} 
\]
    \caption{$S_2$ with $(a_1(S_2),a_2(S_2),a_3(S_2))=(4,4,3)$.}
    \label{fig:s2}
\end{figure}

\end{example}

\begin{lemma}\label{Max2}
Let $S$ be a generalized-$\beta$-set that maximizes $f(S)$. Assume that $A_{t-1}-A_t<2m$, then we have
\begin{enumerate}
    \item[(1)] $s$ is even and $t=\frac{s}{2}$;
    \item[(2)] $a_{(t-1)m}=\cdots=a_{tm-1}=1$ and $a_{tm}=0$.
\end{enumerate}
\end{lemma}

\begin{proof}
By Corollary \ref{cor: 7.3}, we obtain that $S\in\mathbb{F}_m(s)$. Now we prove Conclusion (2). Since $A_{t-1}-A_t<2m$, by Proposition \ref{<2m} we set $a_{(t-1)m}=\cdots=a_{(t-1)m+p}=1>a_{(t-1)m+p+1}$, where $p\leq m-1$.
First, we obtain that ${a_{(t-2)m+1}=\cdots=a_{(t-2)m+p}=3>a_{(t-2)m+p+1}}$ by Lemma \ref{Normal}. Now we claim that ${a_{(t-2)m+p+1}=1}$. 

Otherwise set ${a_{(t-2)m+p+1}=\cdots=a_{(t-2)m+p+q}=2}$, ${a_{(t-2)m+p+q+1}=\cdots=a_{(t-1)m}=1}$. Then $s \ge 2t$ and $q \leq m-p-1$. Let $S'\in\mathbb{E}_m(s)$ be a $t$-row set satisfying the following conditions: 
\begin{enumerate}
\item[(1)] $(n_i',a_i')=(n_i,a_i)$ when $i\not\equiv p+1,p+2,\cdots,p+q\mod{m}$;
\item[(2)] $(n_i',a_i')=(n_i,a_i+1)$ when $i\equiv p+1,p+2,\cdots,p+q\mod{m}$ and $i\leq (t-1)m$;
\item[(3)] $\mathcal{B}_{(t-1)m+p+k}'=\{((t-1)m+p+k)s-t\}$ for $1\leq k\leq q$.
\end{enumerate}
Then obviously $S'\in\mathbb{F}_m(s)$. Set the first $k+1$ elements of $\mathcal{B}_{p+1}'$,$\cdots$,$\mathcal{B}_{p+q}'$ as $l_1,\cdots,l_{q(k+1)}$. Therefore,
\[
|S|=m(t-1)^2+(2p+q)(t-1)+p
\]
and
\begin{eqnarray*}
f(S')-f(S) &=& qms\frac{t(t-1)}{2} + \sum_{i=1}^{qt} l_i-qt|S|-\sum_{i=1}^{qt} (i-1) \\
&=& qt\left(\frac{(t-1)ms}{2}-m(t-1)^2-(2p+q)(t-1)-p\right) + \sum_{i=1}^{qt} (l_i-i+1) \\
& \ge & qt(mt(t-1)-m(t-1)^2-(m+p-1)(t-1)-p)+qt(ps+1) \\ 
&=& q(p+1)t^2>0.
\end{eqnarray*}
A {contradiction}. Thus we can assume that $a_{(t-2)m+j}\in\{1,3\}$ for $1\leq j\leq m$.



\par {Next, we prove ${p=m-1}$ by contradiction}. {Otherwise, we assume that ${p<m-1}$}. Since $a_{(t-1)m+1}=\cdots=a_{(t-1)m+p}=1$ for $1\leq p<m-1$, $s\ge 2t$. Let $S'\in\mathbb{E}_m(s)$ be a $t$-row set satisfying the following conditions: for $1\leq i\leq tm$,
\begin{enumerate}
\item[(1)] $(n_i',a_i')=(n_i,a_i)$ when $i\not\equiv p+1,p+2,\cdots,m-1\mod{m}$;
\item[(2)] $(n_i',a_i')=(n_i+1,a_i+2)$ when $i\equiv p+1,p+2,\cdots,m-1\mod{m}$ and $i\leq (t-1)m$;
\item[(3)] $\mathcal{B}_{(t-1)m+k}'=\{((t-1)m+k)s-t\}$ for $p+1\leq k\leq m-1$.
\end{enumerate}
Obviously $S'\in\mathbb{F}_m(s)$. Set the elements of $\mathcal{B}_{p+1}', \mathcal{B}_{p+2}',\cd,\mathcal{B}_{m-1}'$ as $l_1,\cdots,l_{(m-p-1)(2t-1)}$. 
Therefore, $|S|=m(t-1)^2+2pt-p$, and 
\begin{eqnarray*} 
&  & f\left(S'\right)-f(S) \\&=& (m-p-1)(t-1)^2ms+\sum_{i=1}^{(m-p-1)(2t-1)} l_i -(m-p-1)(2t-1)|S|\\&~~-&\sum_{i=1}^{(m-p-1)(2t-1)} (i-1) \\
&=& (m-p-1)((t-1)^2ms-(2t-1)(m(t-1)^2+2pt-p))\\&~~+&\sum_{i=1}^{(m-p-1)(2t-1)} (l_i-i+1) \\
& \ge & (m-p-1)(2mt(t-1)^2-(2t-1)(m(t-1)^2+2pt-p)) \\&~~+&(m-p-1)(2t-1)(ps+1) \\
& \ge & (m-p-1)(m(t-1)^2-p(2t-1)^2+(2t-1)(2pt+1)) \\
&=& (m-p-1)(m(t-1)^2+(p+1)(2t-1))>0. 
\end{eqnarray*}
A {contradiction}. Thus Conclusion (2) is proved.



\par {
Now we prove Conclusion (1)}. By Lemma \ref{Normal}, $a_{(t-2)m+1}=3$. Let $S'\in\mathbb{E}_m(s)$ be a $t$-row set satisfying the following conditions: for $1\leq i\leq tm$,
\begin{enumerate}
\item[(1)] $(n_i',a_i')=(n_i,a_i+1)$ when $1\leq i\leq (t-1)m$ and $m|i$;
\item[(2)]  $(n_i',a_i')=(n_i,a_i)$ otherwise.
\end{enumerate}
Obviously $S'\in\mathbb{F}_m(s)$.
Since $|S|=m(t-1)^2+2(m-1)(t-1)+(m-1)$, we have
\begin{eqnarray*}
&~&f(S')-f(S) = \frac{(t-1)(t-2)}{2}ms+\sum_{i=1}^{t-1} (ms-2t+i)-(t-1)|S|-\sum_{i=1}^{t-1} (i-1) \\
&=& (t-1)\left(\frac{t-2}{2}ms-m(t-1)^2-2(m-1)(t-1)-(m-1)+ms-2t+1\right) \\
&=&\frac{mt(t-1)(s-2t)}{2}\ge 0.
\end{eqnarray*}
The equality may only be achieved for $s=2t$. Then $s$ is even and $t=\frac{s}{2}$. {Thus Conclusion (1) is proved.}
\end{proof}

There are three kinds of adjustments in Lemma \ref{Max2}. We give an example to help visualize those adjustments.

\begin{example} \label{ex:formax2}

Let $s=6$ and $m=3$. It's easy to verify that $S_{11},S_{12},S_{21},S_{22},S_{31},S_{32}\in\mathbb{F}_m(s)$ in Figure \ref{fig:max2-1-1} to \ref{fig:max2-3-2}. 

Figure \ref{fig:max2-1-1} and Figure \ref{fig:max2-1-2} display the first kind of adjustment. Notice that $S_{12}=S_{11}\cup\{7,26,45\}$. We color $S_{12} \setminus S_{11}$ red. Indeed $p=q=1$, $(n_2(S_{12}), a_2(S_{12}))=(n_2(S_{11}), a_2(S_{11})+1), (n_5(S_{12}), a_5(S_{12}))=(n_5(S_{11}), a_5(S_{11})+1)$ and $\mathcal{B}_8(S_{32})=\{45\}$. We have $f(S_{11})=117<135=f(S_{12})$.

Figure \ref{fig:max2-2-1} and Figure \ref{fig:max2-2-2} display the second kind of adjustment. Notice that $S_{22}=S_{21}\cup\{7,11,26,28,45\}$. We color $S_{22} \setminus S_{21}$ red. Indeed $p=1$, $(n_2(S_{22}), a_2(S_{22}))=(n_2(S_{21})+1, a_2(S_{21})+2), (n_5(S_{22}), a_5(S_{22}))=(n_5(S_{21})+1, a_5(S_{21})+2)$ and $\mathcal{B}_8(S_{22})=\{45\}$. We have $f(S_{21})=113<135=f(S_{22})$.

Figure \ref{fig:max2-3-1} and Figure \ref{fig:max2-3-2} display the third kind of adjustment. Notice that $S_{32}=S_{31}\cup\{15\}$. We color $S_{32} \setminus S_{31}$ red. Indeed $t=2$ and $(n_3(S_{32}), a_3(S_{32}))=(n_3(S_{31}), a_3(S_{31})+1)$. We have $f(S_{31})=72<78=f(S_{32})$.

\begin{figure}[h!]
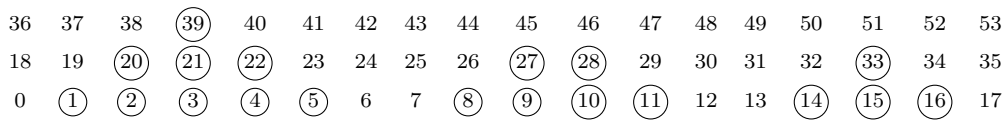

    \centering
    \[ \scriptsize
    \begin{array}{cccccccccccccccccc}
    36 & 37 & 38 & \encircle{39} & 40 & 41 & 42 & 43 & 44 & 45 & 46 & 47 & 48 & 49 & 50 & 51 & 52 & 53 \\
    18 & 19 & \encircle{20} & \encircle{21} & \encircle{22} & 23 & 24 & 25 & 26 & \encircle{27} & \encircle{28} & 29 & 30 & 31 & 32 & \encircle{33} & 34 & 35\\
    0 & \encircle{1} & \encircle{2} & \encircle{3} & \encircle{4} & \encircle{5} & 6 & 7 & \encircle{8} & \encircle{9} & \encircle{10} & \encircle{11} & 12 & 13 & \encircle{14} & \encircle{15} & \encircle{16} & 17\\
    
    \end{array}
    \]
    \caption{$S_{11}$: before the first kind of adjustment.}
    \label{fig:max2-1-1}
\end{figure}

\begin{figure}[h!]
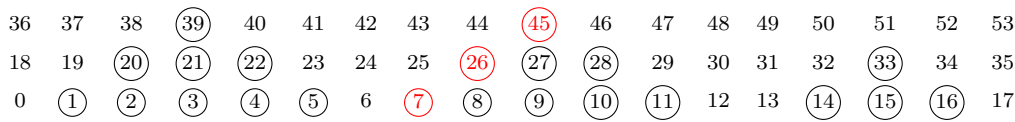

    \centering
    \[ \scriptsize
    \begin{array}{cccccccccccccccccc}
    36 & 37 & 38 & \encircle{39} & 40 & 41 & 42 & 43 & 44 & \enred{45} & 46 & 47 & 48 & 49 & 50 & 51 & 52 & 53 \\
    18 & 19 & \encircle{20} & \encircle{21} & \encircle{22} & 23 & 24 & 25 & \enred{26} & \encircle{27} & \encircle{28} & 29 & 30 & 31 & 32 & \encircle{33} & 34 & 35\\
    0 & \encircle{1} & \encircle{2} & \encircle{3} & \encircle{4} & \encircle{5} & 6 & \enred{7} & \encircle{8} & \encircle{9} & \encircle{10} & \encircle{11} & 12 & 13 & \encircle{14} & \encircle{15} & \encircle{16} & 17\\
    
    \end{array}
    \]
    \caption{$S_{12}$: after the first kind of adjustment.}
    \label{fig:max2-1-2}
\end{figure}

\begin{figure}[h!]
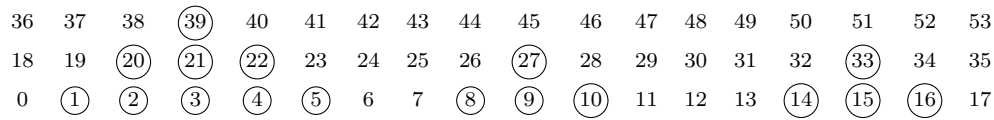

    \centering
    \[ \scriptsize
    \begin{array}{cccccccccccccccccc}
    36 & 37 & 38 & \encircle{39} & 40 & 41 & 42 & 43 & 44 & 45 & 46 & 47 & 48 & 49 & 50 & 51 & 52 & 53 \\
    18 & 19 & \encircle{20} & \encircle{21} & \encircle{22} & 23 & 24 & 25 & 26 & \encircle{27} & 28 & 29 & 30 & 31 & 32 & \encircle{33} & 34 & 35\\
    0 & \encircle{1} & \encircle{2} & \encircle{3} & \encircle{4} & \encircle{5} & 6 & 7 & \encircle{8} & \encircle{9} & \encircle{10} & 11 & 12 & 13 & \encircle{14} & \encircle{15} & \encircle{16} & 17\\
    
    \end{array}
    \]
    \caption{$S_{21}$: before the second kind of adjustment.}
    \label{fig:max2-2-1}
\end{figure}

\begin{figure}[h!]
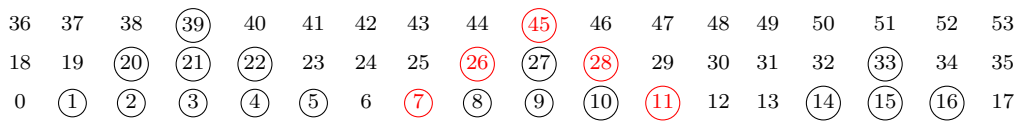

    \centering
    \[ \scriptsize
    \begin{array}{cccccccccccccccccc}
    36 & 37 & 38 & \encircle{39} & 40 & 41 & 42 & 43 & 44 & \enred{45} & 46 & 47 & 48 & 49 & 50 & 51 & 52 & 53 \\
    18 & 19 & \encircle{20} & \encircle{21} & \encircle{22} & 23 & 24 & 25 & \enred{26} & \encircle{27} & \enred{28} & 29 & 30 & 31 & 32 & \encircle{33} & 34 & 35\\
    0 & \encircle{1} & \encircle{2} & \encircle{3} & \encircle{4} & \encircle{5} & 6 & \enred{7} & \encircle{8} & \encircle{9} & \encircle{10} & \enred{11} & 12 & 13 & \encircle{14} & \encircle{15} & \encircle{16} & 17\\
    
    \end{array}
    \]
    \caption{$S_{22}$: after the second kind of adjustment.}
    \label{fig:max2-2-2}
\end{figure}

\begin{figure}[h!]
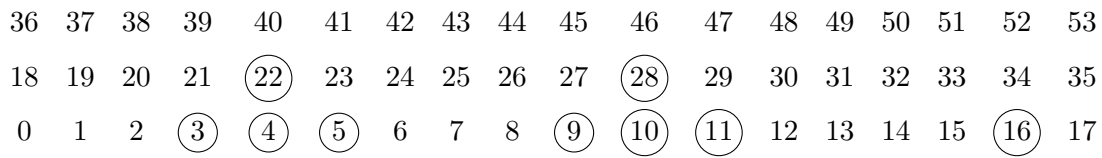

    \centering
    \[  \small
    \begin{array}{cccccccccccccccccc}
    36 & 37 & 38 & 39 & 40 & 41 & 42 & 43 & 44 & 45 & 46 & 47 & 48 & 49 & 50 & 51 & 52 & 53 \\
    18 & 19 & 20 & 21 & \encircle{22} & 23 & 24 & 25 & 26 & 27 & \encircle{28} & 29 & 30 & 31 & 32 & 33 & 34 & 35\\
    0 & 1 & 2 & \encircle{3} & \encircle{4} & \encircle{5} & 6 & 7 & 8 & \encircle{9} & \encircle{10} & \encircle{11} & 12 & 13 & 14 & 15 & \encircle{16} & 17\\
    
    \end{array}
    \]
    \caption{$S_{31}$: before the third kind of adjustment.}
    \label{fig:max2-3-1}
\end{figure}

\begin{figure}[h!]
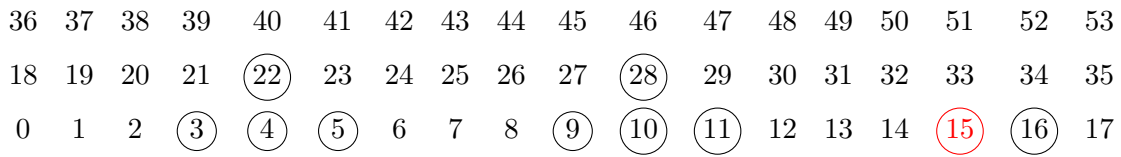

    \centering
    \[  \small
    \begin{array}{cccccccccccccccccc}
    36 & 37 & 38 & 39 & 40 & 41 & 42 & 43 & 44 & 45 & 46 & 47 & 48 & 49 & 50 & 51 & 52 & 53 \\
    18 & 19 & 20 & 21 & \encircle{22} & 23 & 24 & 25 & 26 & 27 & \encircle{28} & 29 & 30 & 31 & 32 & 33 & 34 & 35\\
    0 & 1 & 2 & \encircle{3} & \encircle{4} & \encircle{5} & 6 & 7 & 8 & \encircle{9} & \encircle{10} & \encircle{11} & 12 & 13 & 14 & \enred{15} & \encircle{16} & 17\\
    
    \end{array}
    \]
    \caption{$S_{32}$: after the third kind of adjustment.}
    \label{fig:max2-3-2}
\end{figure}

\end{example}

\par Lemmas \ref{Max1} and \ref{Max2} discuss about the necessary condition for the generalized-$\beta$-set $S$ that maximizes $f(S)$. The following theorem gives a detailed description of $S$. 

\begin{theorem}\label{maximal}
Let $S$ be a generalized-$\beta$-set that maximizes $f(S)$.  If $s$ is odd, then one of the following two cases holds:
\begin{enumerate}
    \item[(1)] $t=\frac{s-1}{2}$ and $a_{(t-1)m+1}=\cdots=a_{tm}=1$;
    \item[(2)] $t=\frac{s-1}{2}$, $a_{(t-1)m+1}=\cdots=a_{tm-1}=2$ and $a_{tm}=1$.
\end{enumerate}If $s$ is even, then one of the following two cases holds:
\begin{enumerate}
    \item[(1)] $t=\frac{s}{2}$ and $a_{(t-1)m}=\cdots=a_{tm-1}=1$;
    \item[(2)] $t=\frac{s}{2}$, $a_{(t-1)m+1}=\cdots=a_{tm-1}=1$ and $a_{(t-1)m}=2$.
\end{enumerate}
\end{theorem}

\begin{proof}

\par {First we discuss about the case where $s$ is odd}. By Corollary \ref{cor: 7.3}, we have $S\in\mathbb{F}_m(s)$. Define $S_k\in\mathbb{F}_m(s) $ as the corresponding generalized-$\beta$-set where $t=k$ and $a_{(t-1)m+1}=\cdots=a_{tm}=1$, $1\leq t\leq \frac{s-1}{2}$; $T_k\in\mathbb{F}_m(s)$ as the corresponding generalized-$\beta$-set where $t=k$ and $a_{(t-1)m+1}=\cdots=a_{tm}=2$, $1\leq t\leq \frac{s-3}{2}$;
$P_k\in\mathbb{F}_m(s)$ as the corresponding generalized-$\beta$-set where $t=k$, $a_{(t-1)m+1}=\cdots=a_{tm-1}=1$ and $a_{tm}=0$, $1\leq t\leq \frac{s-1}{2}$; and $Q_k\in\mathbb{F}_m(s)$ as the corresponding generalized-$\beta$-set where $t=k$, $a_{(t-1)m+1}=\cdots=a_{tm-1}=2$ and $a_{tm}=1$, $1\leq t\leq \frac{s-1}{2}$. By the discussion in Lemma \ref{Max1}, we only need to compare $f(S_k),f(T_k),f(P_k)$ and $f(Q_k)$.

\par Given that $|S_k|=mk^2$, for $1\leq k\leq \frac{s-3}{2}$, we have
\begin{eqnarray*}
f(S_{k+1})-f(S_k) &=& \sum_{x\in S_{k+1}}x -\sum_{x\in S_{k}} x -\sum_{i=|S_k|}^{|S_{k+1}|-1}i \\
&=& (ms-1)mk^2+\sum_{i=1}^{m} \sum_{j=2}^{2k+2}(is-j) - m(2k+1)mk^2-\sum_{i=0}^{m(2k+1)-1}i \\
&\ge& mk^2(m(s-2k-1)-1)+\sum_{i=0}^{m(2k+1)-1}(s-2k-2+i)-\sum_{i=0}^{m(2k+1)-1}i \\
&>& 0.
\end{eqnarray*}
Therefore, $\left\{S_{\frac{s-1}{2}}\right\}=\arg\max_{1\leq k\leq (s-1)/2} f(S_k)$. With similar calculations, we have
\[
\left\{T_{\frac{s-3}{2}}\right\}=\arg\max_{1\leq k\leq (s-3)/2} f(T_k);
\]
\[
\left\{P_{\frac{s-1}{2}}\right\}=\arg\max_{1\leq k\leq (s-1)/2} f(P_k);
\]
\[
\left\{ Q_{\frac{s-1}{2}} \right\}=\arg\max_{1\leq k\leq (s-1)/2} f(Q_k).
\]

\par Now we compare $f\left(S_{\frac{s-1}{2}}\right),f\left(T_{\frac{s-3}{2}}\right),f\left(P_{\frac{s-1}{2}}\right),f\left(Q_{\frac{s-1}{2}}\right)$. 

\par For $1\leq k\leq \frac{s-3}{2}$, set the first $k+1$ elements of each one of $\mathcal{B}_1(S_{k+1})$,$\cdots$,$\mathcal{B}_m(S_{k+1})$ as $l_1,\cdots,l_{m(k+1)}$. Then we have 
\begin{eqnarray*}
f(S_{k+1})-f(T_k) &=& \sum_{x\in S_{k+1}}x-\sum_{x\in T_k}x-\sum_{i=|T_k|}^{|S_{k+1}|-1}i \\
&=& ms\cdot \frac{1}{2}mk(k+1)+\sum_{i=1}^{m(k+1)}l_i-m(k+1)|T_k|-\sum_{i=1}^{m(k+1)}(i-1) \\
&=& m^2k(k+1)\left(\left(\frac{s}{2}-k-1\right)+\sum_{i=1}^{m(k+1)}\left(l_i-i+1\right)\right) \\
&>& 0.
\end{eqnarray*}
Let $k=\frac{s-3}{2}$. Then  $f\left(S_{\frac{s-1}{2}}\right)>f\left(T_{\frac{s-3}{2}}\right)$.

\par For $1\leq k\leq \frac{s-1}{2}$, set the first $k$ elements of each one of $\mathcal{B}_1(Q_k)$,$\cdots$,$\mathcal{B}_m(Q_k)$ as $l_1,\cdots,l_{mk}$. Given that $|P_k|=mk^2-k$ and $|Q_k|=mk^2+(m-1)k$, we have
\begin{eqnarray*}
&~&f(Q_k)-f(P_k) \\&=& \sum_{x\in Q_k}x-\sum_{x\in P_k}x-\sum_{i=|P_k|}^{|Q_k|-1}i \\
&=& ms\frac{m}{2}k(k-1)+\sum_{i=1}^{mk}l_i-mk(mk^2-k)-\sum_{i=0}^{mk-1}i \\
&\ge& \frac{m^2 s}{2}k(k-1)+m\frac{k(k+1)}{2}+\frac{m(m-1)}{2}ks-mk(mk^2-k)-\frac{mk(mk-1)}{2} \\
&=& \frac{mk}{2}((mk-1)s-2mk^2+(3-m)k+2) \\
&\ge& \frac{mk}{2}((mk-1)(2k+1)-2mk^2+(3-m)k+2) \\
&=& \frac{mk(k+1)}{2}>0.
\end{eqnarray*}
Let $k=\frac{s-1}{2}$. Then $f\left(Q_{\frac{s-1}{2}}\right)>f\left(P_{\frac{s-1}{2}}\right)$.

\par For $1\leq k\leq \frac{s-1}{2}$, since $|S_k|=mk^2$ and $|Q_k|=mk^2+(m-1)k$, we have
\begin{eqnarray*}
&~&f(Q_k)-f(S_k) \\&=& \sum_{x\in Q_k}x-\sum_{x\in S_k}x-\sum_{i=|S_k|}^{|Q_k|-1}i \\
&=& ms(m-1)\frac{k(k-1)}{2}+(m-1)\left(ks-\frac{k(k+1)}{2}\right)+\frac{(m-1)(m-2)}{2}ks \\
&& -(m-1)kmk^2-\frac{(m-1)k((m-1)k-1)}{2} \\
&=& \frac{m(m-1)k^2}{2}(s-2k-1) \\
&\ge& 0.
\end{eqnarray*}
Let $k=\frac{s-1}{2}$. Then $f\left(Q_{\frac{s-1}{2}}\right)=f\left(S_{\frac{s-1}{2}}\right)$.
From the discussion above, we know that $S_{\frac{s-1}{2}}$ and $Q_{\frac{s-1}{2}}$ are the only two generalized-$\beta$-sets that maximize $f(S)$.

\par Next, {we discuss about the case where $s$ is even}. We first consider the case where $A_{t-1}-A_t\ge 2m$. Then by Lemma \ref{Normal}, $a_i-a_{i+m}=2$ for $1\leq i\leq (t-1)m$. Just like the former case, for $1\leq k\leq \frac{s}{2}-1$, we can define $S_k,T_k,Q_k$; for $1\leq k\leq \frac{s}{2}$, we can define $P_k$. With similar discussions, we only need to compare $f\left(S_{\frac{s}{2}-1}\right),f\left(T_{\frac{s}{2}-1}\right),f\left(P_{\frac{s}{2}}\right),f\left(Q_{\frac{s}{2}-1}\right)$.

\par Since $f(Q_k)-f(S_k)=\frac{m(m-1)k^2}{2}(s-2k-1)$, we have $f\left(Q_{\frac{s}{2}-1}\right)>f\left(S_{\frac{s}{2}-1}\right)$.
For $1\leq k\leq \frac{s}{2}-1$, since $|P_{k+1}|=mk^2+(2m-1)k+(m-1)$ and $|T_k|=mk(k+1)$, we have
\begin{align*}
f(P_{k+1})-f(T_k) &= \sum_{x\in P_{k+1}}x-\sum_{x\in T_k}x-\sum_{i=|T_k|}^{|P_{k+1}|-1}i \\
&= (m-1)\frac{k(k+1)}{2}ms+\sum_{i=1}^{m-1}\sum_{j=1}^{k+1}(is-j)-(m-1)(k+1)mk(k+1) \\
&\ \ \ -\sum_{i=0}^{(m-1)(k+1)-1}i\\
&\ge \frac{mk}{2}(m-1)(k+1)(s-2k-2)+\sum_{i=0}^{(m-1)(k+1)-1}(s-k-1+i) \\
&\ \ \ -\sum_{i=0}^{(m-1)(k+1)-1}i\\
&\ge (m-1)(k+1)\frac{s}{2}>0.
\end{align*}
Thus we have $f\left(P_{\frac{s}{2}}\right)>f\left(T_{\frac{s}{2}-1}\right)$. Similarly, we can prove that $f\left(P_{\frac{s}{2}}\right)>f\left(Q_{\frac{s}{2}-1}\right)$.

\par Therefore, $S=P_{\frac{s}{2}}$ is the only generalized-$\beta$-set that maximizes $f(S)$ when $A_{t-1}-A_t\ge 2m$. Let $P_{\frac{s}{2}}'$ be the corresponding generalized-$\beta$-set such that $t=\frac{s}{2}$, $a_{(t-1)m}=\cdots=a_{tm-1}=1$ and $a_i-a_{i+m}=2$ for $1\leq i\leq (t-1)m-1$. By Lemma \ref{Max2}, $f\left(P_{\frac{s}{2}}\right)=f\left(P_{\frac{s}{2}}'\right)$, and we obtain that~$P_{\frac{s}{2}}$ and $P_{\frac{s}{2}}'$ are the only two generalized-$\beta$-sets that maximize $f(S)$.
\end{proof}

\begin{example}
Figures \ref{L_3(5)-2} and  \ref{P_2} give an example where $s$ is odd. Figures \ref{L_3(6)-2} and  \ref{P_3'} give an example where $s$ is even.
\end{example}

Finally, we can give the proof of the main result Theorem \ref{conj}.

\begin{proof}[Proof of Theorem \ref{conj}]
\par By Lemma \ref{if}, we only need to consider the generalized-$\beta$-sets $S$ that maximize $f(S)$. 
\par When $s$ is odd, $Q_{\frac{k-1}{2}}=\mathcal {L}_m(s)$. By Theorem \ref{maximal}, $S=Q_{\frac{k-1}{2}}$ or $S=S_{\frac{k-1}{2}}$. Both of them satisfy the conditions in Theorem \ref{L_m(s)}. Thus they are both $\beta$-sets of certain $(s,ms-1,ms+1)$-core partitions. Let $M(s,m)$ be the corresponding partition of $\mathcal {L}_m(s)$. Since the conjugate of a partition has the same size as itself, by Lemma \ref{conjugate}, $S_{\frac{k-1}{2}}$ is the $\beta$-set of the conjugate of $M(s,m)$.
\par Similarly, the conclusion is also true when $s$ is even.
\end{proof}

\section{Discussions}
By extending the concept of $\beta$-sets to generalized-$\beta$-sets, 
we determine the possible structures of $(s,ms-1,ms+1)$-core partitions  with the largest size, which proves the conjecture proposed by Nath and Sellers in \cite{NS2}. 
Currently, most research results on simultaneous core partitions are about {$(s_1,s_2,\ldots, s_m)$}-core partitions with $m\leq 3$. When $m\geq 4$, it still lacks proper tools for studying such general simultaneous core partitions.
We believe that the concepts and techniques related to generalized-$\beta$-sets introduced in this paper offer some insights for exploring statistics of general simultaneous core partitions, of which we know very little at this moment.

\section*{Acknowledgements}

This work was supported by the National Science Foundation of China [Grant No. 12201155].

\end{document}